\documentclass[10pt,reqno,twoside]{amsart}
\usepackage{graphicx}
\usepackage{amssymb}
\usepackage{epstopdf}
\usepackage[asymmetric,top=3.5cm,bottom=4.3cm,left=3.1cm,right=3.1cm]{geometry}
\usepackage{graphicx}
\usepackage{bm}

 \usepackage{booktabs} 
 \usepackage{array} 
 \usepackage{paralist} 
 \usepackage{verbatim} 
 \usepackage{subfigure} 
 \usepackage{tabularx}

 \usepackage{tabularx}
\usepackage{amsmath,amsfonts,amsthm,mathrsfs,amssymb,cite}
\usepackage[usenames]{color}
 \usepackage{hyperref}
 \usepackage{cleveref}

\newtheorem{thm}{Theorem}[section]

\newtheorem{lem}{Lemma}[section]
\newtheorem{prop}{Proposition}[section]
\theoremstyle{definition}
\newtheorem{defn}{Definition}[section]
\theoremstyle{remark}

\newtheorem{rem}{Remark}[section]
\numberwithin{equation}{section}

\newcommand{\bu}{\mathbf{u}}

\newcommand{\bv}{\mathbf{v}}
\newcommand{\bw}{\mathbf{w}}
\newcommand{\bff}{\mathbf{f}}

\newcommand{\bx}{\mathbf{x}}

\newcommand{\bd}{\mathbf{d}}

\newcommand{\bh}{\mathbf{h}}
\newcommand{\bt}{\mathbf{t}}

\newcommand{\bg}{\mathbf{g}}

\newcommand{\eeta}{\bm{\eta}}
\newcommand{\bxi}{\bm{\xi}}

\def\Oh{{\mathcal O}}

\newcommand{\beq}{\begin{equation}}
\newcommand{\eeq}{\end{equation}}

\title[Dislocations in a multi-layered elastic solid]
{Dislocations in a multi-layered elastic solid with applications to  fault and interface identifications}

\author{Huaian Diao}
\address{School of Mathematics, Jilin University and Key Laboratory of Symbolic Computation and Knowledge Engineering of Ministry of Education, Changchun, Jilin, China.}
\email{diao@jlu.edu.cn, hadiao@gmail.com}
\author{Hongyu Liu}
\address{Department of Mathematics, City University of Hong Kong, Kowloon Tong, Hong Kong SAR, China.}
\email{hongyu.liuip@gmail.com; hongyliu@cityu.edu.hk}
\author{Qingle Meng}
\address{Department of Mathematics, City University of Hong Kong, Kowloon Tong, Hong Kong SAR, China.}
\email{mengql2021@foxmail.com; qinmeng@cityu.edu.hk}
\date{}
\begin{document}

\maketitle
\begin{abstract}
This paper investigates an elastic dislocation problem within a bounded and multi-layered solid governed by the Lam\'e system. We address the simultaneous reconstruction of the faults, the jumps in displacement and traction fields across the faults, and the interfaces of layers using a single passive boundary measurement. This inverse problem is particularly challenging due to the discontinuities in both the displacement and traction fields across the faults and the inherent difficulty of establishing uniqueness results with limited measurement data. Under the assumptions that the Lam\'e parameters are piecewise constants within each layer, satisfying strong convexity conditions, and that the faults exhibit corner singularities, we establish local uniqueness identifiability results for both the interfaces and the faults, as well as the jumps across the faults. Furthermore, we derive global uniqueness results for reconstructing the interfaces, the faults, and the corresponding displacement and traction jumps in generic scenarios under a priori geometric information, where the faults are geometrically general and may be either open or closed.
\medskip


\medskip

\end{abstract}


\section{Introduction}
\subsection{Mathematical setup}
 Initially focusing on mathematics, but not physics and geography applications, we begin by
introducing the mathematical setup of the study. Let $\Omega$ be a bounded Lipschitz domain in $ \mathbb{R}^n$ with $n = 2, 3$,  composed of elastic materials. The Lam\'e parameters, denoted by $\lambda $ and $\mu$, characterize the elastic properties of the materials in $\Omega$, which are positive real-valued functions and belong to $L^\infty(\Omega)$.
 The elasticity tensor for $\Omega$ is denoted by $\mathcal{C}(\mathbf{x})$, and its components are defined as follows:
\begin{equation*}
\mathcal{C}_{ijkl}(\mathbf{x})=\lambda(\mathbf{x})\delta_{ij}\delta_{kl}+\mu(\mathbf{x})(\delta_{ik}\delta_{jl}+\delta_{il}\delta_{jk}),
\end{equation*}
where $\delta$ represents the Kronecker delta. The parameters $\lambda$ and $\mu$ are assumed to satisfy the strong convexity conditions:
\begin{equation}\label{tensor2}
\mu(\mathbf{x})> 0 \quad \mbox{and} \quad 2\mu(\mathbf{x})+n\lambda(\mathbf{x})>0,\quad \forall \,\,\mathbf{x}\,\in\Omega.
\end{equation}


Assume that there exists an oriented Lipschitz curve or surface, denoted by $\Sigma \Subset \Omega$, that models a buried dislocation, also referred to as a fault. Across this fault, discontinuities may occur, resulting in jumps in both displacement and traction fields between the two sides of $\Sigma$, denoted as $\Sigma^\pm$. The jump associated with the displacement field is referred to as a slip. Here,  $\Sigma^+$ represents the side where the unit normal vector $\nu$ points towards the boundary $\partial \Omega$, while $\Sigma^-$ represents the opposite side.  The jumps across $\Sigma$ are expressed as follows:
\begin{equation*}
\begin{cases}
[\bu]_\Sigma=\bu\big|_{\Sigma^+}-\bu\big|_{\Sigma^-},\\
[\mathcal{T}_\nu\mathbf{u}]_\Sigma=\mathcal{T}_\nu\mathbf{u}\big|_{\Sigma^+}-\mathcal{T}_\nu\mathbf{u}\big|_{\Sigma^-},
\end{cases}
\end{equation*}
where
\begin{equation*}
\begin{cases}
\mathcal{T}_\nu\mathbf{u}=\lambda(\nabla\cdot\mathbf{u})\nu+2\mu(\nabla^s\mathbf{u}),\\
\bu\big|_{\Sigma^\pm}=\lim\limits_{\varepsilon\rightarrow 0^+}\bu\big(\bx\pm \varepsilon\nu(\bx)\big),\\
\mathcal{T}_\nu\mathbf{u}\big|_{\Sigma^\pm}=\lim\limits_{\varepsilon\rightarrow 0^+}\mathcal{T}_\nu\mathbf{u}\big(\bx\pm \varepsilon\,\nu(\bx)\big).
\end{cases}
\end{equation*}
Here, $\nabla^s\mathbf{u}:=\frac 1 2(\nabla\mathbf{u}+\nabla\mathbf{u}^t) $, with the superscript $``\,t\,"$  indicating the matrix transpose.

This paper considers a generic elastic PDE system that permits jumps in the displacement and traction fields across \(\Sigma\) within a priori function spaces (see Subsection \ref{sub:Functional}). In the context of the direct problem, our goal is to find $\mathbf{u}\in H^1_{\partial \Omega_2}(\Omega\backslash\overline{\Sigma})^n$ satisfying
\begin{equation}\label{eq:elast1}
\begin{cases}
\Delta^*\mathbf{u}(\mathbf{x})+\omega^2\,\mathbf{u}=\mathbf{0},\quad\mathbf{x}\in\Omega\backslash\overline{\Sigma},\\
\mathcal{T}_\nu\mathbf{u}|_{\partial \Omega_2}=\mathbf{0},\quad \mathbf{u}|_{\partial \Omega_1}=\mathbf{0},\\
[\mathbf{u}]_\Sigma=\mathbf{f},\,\,\quad\quad [\mathcal{T}_\nu \mathbf{u}]_\Sigma=\mathbf{g},
\end{cases}
\end{equation}
where $\omega\in\mathbb{R}_+$ denotes the angular frequency of the elastic displacement. Additionally, $\partial \Omega_1$ and $\partial \Omega_2$ form a Lipschitz dissection of $\partial \Omega$ (i.e., $\partial\Omega=\partial\Omega_1\cup\partial\Omega_2$ and $\partial \Omega_1\cap\partial \Omega_2=\emptyset$). The Lam\'e operator $\Delta^*(\cdot)$ is defined as follows:
\begin{equation}\label{def:Lam}
\Delta^* \mathbf{u}:=\mu\Delta \mathbf{u}+(\lambda+\mu)\nabla\nabla\cdot \mathbf{u}.
\end{equation}
The elastic model \eqref{eq:elast1} can be obtained by considering the time-harmonic elastic scattering of surface sources in \(\mathbb{R}^n \setminus \Sigma\) (see \cite{PowW2021}), which can be established by introducing the Dirichlet-to-Neumann map for truncating the unbounded domain \(\mathbb{R}^n \setminus \Sigma\) to a bounded one. In this paper, the domain \(\Omega\) is considered to be partitioned into a finite number of disjoint subdomains \(\Omega_1, \Omega_2, \ldots, \Omega_N\), where \(N \in \mathbb{N}\) and \(N \geq 1\). Specifically,
\begin{equation}\label{Part}
\Omega = \bigcup_{l=1}^N \Omega_l\quad \text{and}
\quad \begin{cases}
\overline{\Omega}_i \cap \overline{\Omega}_{j}=\emptyset& \text{when}\quad j-i\geq 2,\\[2mm]
\overline{\Omega}_i \cap \overline{\Omega}_{j}=\Gamma_i
& \text{when}\quad j-i= 1,
\end{cases}
\end{equation}
where $\Gamma_i$ is called the interface between subdomains $\Omega_i$ and $\Omega_{i+1}$ for $i = 1, \ldots, N-1$. In addition, $\partial \Omega_0\Subset\partial \Omega_1\cap\partial \Omega$; see Fig.~\ref{Fig1} for a schematic illustration.

For the solution $\bu$ to Problem~\eqref{eq:elast1}, introduce the mapping defined as follows:
\begin{equation}\label{eq:measurement}
\mathcal{F}_{\Sigma, \mathbf{f},\mathbf{g},\Gamma_1,\ldots, \Gamma_{N-1}}=\mathbf{u}\big|_{\partial \Omega_0},
\end{equation}
where $\partial\Omega_0$ is a nonempty subset of $\partial \Omega_2$.   Indeed, $\mathcal{F}_{\Sigma;\mathbf{f},\mathbf{g},\Gamma_l}$ encodes the elastic deformation data induced by the fault $\Sigma$, along with the jumps $\mathbf{f}$ and $\mathbf{g}$, as observed on $\partial\Omega_0$. For the inverse problem, we aim to determine the fault $\Sigma$, the interfaces $\Gamma_1,\dots,\Gamma_{N-1}$, and the jumps $\mathbf{f}$ and $\mathbf{g}$ using knowledge of a single passive measurement on $\partial \Omega_0$ given by the mapping \eqref{eq:measurement}. Specifically,
\begin{equation}\label{eq:ip1}
\mathcal{F}_{\Sigma;\mathbf{f},\mathbf{g},\Gamma_1, \ldots, \Gamma_{N-1}}\longrightarrow \Sigma, \mathbf{f}, \mathbf{g},\ldots, \Gamma_{N-1}.
\end{equation}

The local and global uniqueness results presented in this paper apply to scenarios where the fault $\Sigma$ can be either open or closed. The main contributions of this paper are summarized as follows:

\begin{itemize}
    \item \textbf{Local Unique Identifiability of Interfaces and Faults}:
    Assuming that the Lam\'e parameters $\lambda$ and $\mu$ are piecewise constants within each layer and satisfy the strong convexity condition \eqref{tensor2}, we prove that the symmetric difference $\Gamma^1_i \Delta \Gamma^2_i := (\Gamma^1_i \setminus \Gamma^2_i) \cup (\Gamma^2_i \setminus \Gamma^1_i)$ for $i = 1, 2, \dots, N-1$ between two distinct interfaces $\Gamma^j_1, \Gamma^j_2, \dots, \Gamma^j_{N-1}$ ($j = 1, 2$) of $\Omega$ cannot exhibit corner singularities, including planar and 3D edge corners. Furthermore, Theorem \ref{th:main_loca} establishes the local unique identifiability of the faults. Under the generic conditions specified in Definitions \ref{def:Admis1-2} and \ref{def:Admis1-3}, the difference between two distinct faults $\Sigma_1$ and $\Sigma_2$ also lacks corner singularities.

    \item \textbf{Global Uniqueness with A Priori Information}:
    If a priori information is available regarding the interfaces $\Gamma_1, \Gamma_2, \ldots, \Gamma_{N-1}$, specifically that they are piecewise polygonal, these interfaces can be uniquely identified using a single passive measurement on $\partial \Omega_0$. Moreover, in practical scenarios where the open fault is modeled as a polygonal curve or surface, and the closed fault is the boundary of a convex polygon or polyhedron, we demonstrate that, under the generic conditions outlined in this work, the fault (which may span multiple layers of the bounded domain) can be uniquely determined with only one measurement. For further details, see Theorems \ref{th:main1} and \ref{th:main2}.
\end{itemize}

\subsection{Connection to existing studies and discussion}

The study of elastic dislocation problems has long been a central focus across a wide range of scientific and engineering disciplines, including crystalline materials, geophysics, and seismology; see \cite{Ar2002book, BM1985, M1949} and the references therein for comprehensive discussions. In classical elastic dislocation models, displacement fields are typically assumed to exhibit discontinuities, while traction fields are considered continuous. This assumption simplifies the relationship between displacement and traction fields to a linear form, as demonstrated in \cite{CS1995, M1949}. Extensive theoretical research has been conducted on elastic dislocation problems in both half-space domains \cite{Volkov2017, VS2019, TV2020, ABMH2020} and bounded domains \cite{ABM2022, ABHM2022, BF2006, BFV2008}, with a primary focus on identifying faults and their associated displacement fields from boundary measurements.

Numerical studies in unbounded domains have utilized a variety of approaches, including a two-step algorithm based on a nonlinear quasi-Newton method \cite{AS1994}, stochastic and statistical techniques \cite{EM2012, FW2008, VS2019}, and iterative methods derived from Newton’s approach \cite{Volkov2017}. In contrast, numerical research in bounded domains can be founded in \cite{BFKL2010} for the reconstruction for a linear crack. Volkov et al. \cite{Volkov2017} introduced iterative methods for detecting fault planes and tangential slips using limited surface measurements. Further progress in the field includes the development of Bayesian approaches in \cite{VS2019} to quantify uncertainties in reconstructions and the derivation of stability estimates for constant Lam\'e coefficients in \cite{TV2020}. For elastic dislocation problems where traction fields are continuous across the faults and only displacement fields exhibit jumps, the authors in \cite{ABM2022, ABHM2022} demonstrated well-posedness through variational methods by introducing functional spaces with favorable extension properties. Furthermore, they established unique results for determining the fault geometry and the slip distribution based on a single boundary measurement.

Traction fields in elastic dislocation problems can exhibit discontinuities across faults, as explored in \cite{M1949, BM1985, Pichierri2020, Ar2002book}. For example, the two-dimensional linear elastic model proposed in \cite{Pichierri2020} demonstrates the simultaneous occurrence of discontinuities in both displacement and traction fields along a fault modeled as a straight interface. Open faults, such as creeping faults, have garnered significant attention due to their important implications in geophysics \cite{ABM2022}. In contrast, closed faults in anisotropic and inhomogeneous elastic media have been rigorously analyzed, with well-posedness established using integral equation methods \cite{PowW2021}. In \cite{DLM2023}, we investigated elastic dislocations in isotropic, homogeneous, and bounded domains using a single boundary measurement. Our model incorporates discontinuities in both displacement and traction fields along fault surfaces characterized by corner singularities, encompassing both open and closed faults. For the direct problem associated with classical elastic dislocations, extensive studies have been conducted in both bounded and unbounded domains. Well-posedness results for these problems can be found in \cite{Zwieten2014, Costabel1993, Mazzucato2010, Dauge1988, Nicaise1992, Li2013}.

In this work, we investigate the dislocation problem in a layered and bounded elastic domain. Compared to existing results (see, e.g., \cite{Volkov2017, VS2019, TV2020, ABMH2020, ABM2022, ABHM2022, BF2006, BFV2008}), the model in \eqref{eq:elast1} is more complex, involving jumps in both the displacement and traction fields along the fault. The fault may be open or closed and is characterized by corner singularities, including both 2D corners and 3D edge corners. The well-posedness of Problem \eqref{eq:elast1} can be established through a variational approach, analogous to the methodology outlined in \cite{DLM2023}. As highlighted in \cite{CK, CK2018, Isa2017}, achieving uniqueness in inverse scattering problems with a single measurement remains a significant challenge, as existing uniqueness results typically require a priori geometric information. We refer to \cite{CK, CK2018, Isa2017} and the references therein for a comprehensive discussion of these developments. Similar challenges arise in the inverse problem associated with dislocation scenarios. To address this, it is necessary to incorporate a priori information regarding the geometry of the faults, the regularity of the displacement jumps, and assumptions about the interfaces. These considerations are crucial for establishing unique identifiability results with a single measurement, as formalized in Definitions \ref{def:Admis1-3}--\ref{def:poly}. Importantly, our analysis encompasses scenarios where the fault may be open or closed and extends across multiple layers within the bounded domain.

Unlike existing processing techniques, this paper conducts a local singularity analysis of the elastic field around the corners of faults using a microlocal approach. This allows us to characterize the jumps in the displacement and traction fields across these faults. To achieve this, we impose certain H\"older regularity assumptions on the jumps around the underlying corners, as detailed in Definitions \ref{def:Admis1-2} and \ref{def:Admis1-3}. In the 3D case, an additional assumption is necessary: both jumps must be independent of one spatial variable, as clarified in Definition \ref{def:Admis1-3}.
We employ Complex Geometric Optics (CGO) solutions for the underlying elastic system to characterize these jump functions at the 2D corner points and 3D edge points. Our study involves intricate and delicate analysis, leading to unique results for the inverse problem of determining the faults, their associated jumps, and the interfaces from a single passive measurement. These results hold regardless of whether the fault curves or surfaces are open or closed.


The remaining sections of the paper are organized as follows. In Section \ref{Main results}, we present the preliminary spaces, the admissible assumptions for the inverse problems, and state our main results. In Section \ref{se:pre}, we derive several local results for the jump vectors $\bff$ and $\bg$ at the corner points on the fault $\Sigma$.  Section \ref{proofs} is dedicated to proving the uniqueness of the inverse problem as outlined in Theorems \ref{th:main_loca}-\ref{th:main2}.

\par
\section{Preliminaries and statements of the main results
}\label{Main results}
\subsection{Functional space settings}\label{sub:Functional}

The fault $\Sigma$ of interest may be either open or closed, which implies that the regularity requirements for the jumps $\bff$ and $\bg$ may vary depending on the specific context. Subsequently, we outline several related function spaces. The notation $H^s(\Sigma)^n$ denotes the standard Sobolev space of vector type, defined as the $L^2$-based Sobolev space with regularity index $s\in \mathbb{R}$ and the corresponding dual space is denoted by $H^{-s}(\Sigma)^n$. Additionally,
$C^\infty(\Sigma)$  represents the space of smooth scalar functions within the domain $\Sigma$, while  $C^\infty_0(\Sigma)^n$ denotes the space of smooth vector functions with compact support in $\Sigma$. When $s\geq 0$, $H_{0}^{s}(\Sigma)^n$ denotes the closure of $ C_0^\infty(\Sigma)^n$ with respect to the norm $\|\cdot\|_{H^{s}(\Sigma)^n}$.

In this paper, we focus exclusively on the regularity index $s=1/2$. Indeed, there exists a specific issue regarding the continuous extension of $H^{1/2}_0(\Sigma)$  to $H^{1/2}_0(\mathcal{\mathbb{R}}^n)$  when $\Sigma$ is open. It is a well-known fact that $H_{0}^{1/2}(\Sigma)^n=H^{1/2}(\Sigma)^n$ in this case. Following the method  described in (cf. \cite{ABHM2022, ABM2022, Lions1973}), we introduce the so-called {\it Lions-Magenes} space \cite{Lions1973} and its dual space, denoted by $H_{00}^{\frac{1}{2}}(\Sigma)^n$ and $ H^{-\frac{1}{2}}_0(\Sigma)^n$, respectively. For the reader's convenience, we briefly review this space. For a more detailed discussion, please refer to \cite{Mclean2010, Tartar2006, Cessenat1998}. The weighted space mentioned above is defined as follows:
\begin{align*}
H_{00}^{1/2}(\Sigma)^n=\left\{ \bu\in H_{0}^{1/2}(\Sigma)^n;\,\,\varrho^{-1/2} \bu\in L^2(\Sigma)^n      \right\}
\end{align*}
associated with the norm
\begin{alignat*}{2}
\|\bu\|_{H_{00}^{1/2}(\Sigma)^n}&=\|\bu\|_{H^{1/2}(\Sigma)^n}+\|\varrho^{-1/2}\bu\|_{L^2(\Sigma)^n}
&\,\quad\mbox{for}\,\,&\bu\,\in H_{00}^{1/2}(\Sigma)^n,
\end{alignat*}
where $\varrho\in C^\infty(\overline{\Sigma})$denotes a weight function with the following properties: (1) $\varrho$ has the same order as the distance to the boundary (i.e., $\lim\limits_{\bx\rightarrow \bx_0}\frac{\varrho(\bx)}{\mathrm{d}(\bx,\partial \Sigma)}=d\neq 0,\, \forall \,\bx_0\,\in \partial \Sigma$); and (2) $\varrho(\bx)$ is positive in $\Sigma$ and $\varrho$ vanishes on $\partial \Sigma$.

Let $\widetilde{\Sigma}$ be a closed Lipschitz curve or surface extending $\Sigma$ and satisfying $\widetilde{\Sigma}\cap \partial \Omega=\emptyset$, where $\widetilde{\Sigma}=\overline{\Sigma}\cup\Sigma_0$. Here,  $\Sigma_0$ is a curve or surface that connects to the boundary $\partial\Sigma$ satisfies  $ \Sigma_0\cap( \Sigma\backslash \partial \Sigma) =\emptyset $. Furthermore, let $\bff$ be continuously extended to $\widetilde{\bff}\in H^{1/2}(\widetilde{\Sigma})^n$ by setting it to zero on $\widetilde{\Sigma} \backslash \Sigma$. Similarly, let $\bg$ be continuously extended to $\widetilde{\bg}\in H^{-\frac{1}{2}}(\widetilde{\Sigma})^n$  by setting it to zero on $\widetilde{\Sigma} \backslash \Sigma$. As discussed in \cite{ABHM2022, ABM2022, Tartar2006}, we know that $H^{1/2}_{00}(\Sigma)$ is the optimal subspace of the space $H^{1/2}(\widetilde{\Sigma})$, whose elements can be continuously extended by zero to a component of $H^{1/2}(\widetilde{\Sigma})$.

Thus, the functional setting for the jumps $\bff$ and $\bg$ is summarized as follows:
\begin{itemize}
	\item If $\Sigma$ is open, then $\mathbf{f}\in H_{00}^{1/2}(\Sigma)^n$ and $ \mathbf{g}\in H_0^{-1/2}(\Sigma)^n$;
\item If $\Sigma$ is closed, then $\mathbf{f}\in H^{1/2}(\Sigma)^n$ and $\mathbf{g}\in H^{-1/2}(\Sigma)^n$.
\end{itemize}

The following theorem establishes that Problem \eqref{eq:elast1} is well-posed. This result is derived using a variational approach similar to the one presented in \cite{DLM2023}.
\begin{thm}
There exists a unique solution $\bu \in H^1_{\partial \Omega_2}(\Omega \setminus \overline{\Sigma})^n$ to Problem~\eqref{eq:elast1}.
\end{thm}
\begin{proof}
Since the proofs of the well-posedness of Problem \eqref{eq:elast1} for both open and closed cases of $\Sigma$ are similar, we will focus exclusively on the case where $\Sigma$ is open. We employ arguments analogous to those in the proof of Theorem 2.1 from \cite{DLM2023}, incorporating necessary modifications.

We consider the variational formulation of Problem \eqref{eq:elast1}: find $\bw\in H^1_{\partial \Omega_2}(\Omega)^n$ such that
$$
a(\bw,\mathbf{\phi})=a_1(\bw,\mathbf{\phi})+a_2(\bw,\mathbf{\phi})=\mathcal{F}(\mathbf{\phi}),\qquad \forall\,\,\overline{\mathbf{\phi}}\,\,\in H^1_{\partial \Omega_2}(\Omega \setminus \overline{\Sigma})^n.$$
Here,
\begin{align*}
a_1(\bw,\mathbf{\phi})&:=\int_{\Omega_1}(\mathcal{C}:\nabla \bw):\nabla\overline{\mathbf{\phi}}\,\mathrm d\mathrm \bx+\int_{\Omega^c_1}(\mathcal{C}:\nabla \bw):\nabla\overline{\mathbf{\phi}}\,\mathrm d\mathrm\bx+\int_{\Omega_1}\, \omega^2\,\mathbf{w}\cdot \overline{\mathbf{\varphi}}\,\mathrm d\mathbf{x} +\int_{\Omega^c_1}\, \omega^2\,\mathbf{w}\cdot \overline{\mathbf{\varphi}}\,\mathrm d \mathrm\bx,\\
 a_2(\mathbf{w},\mathbf{\phi})&:=-2\int_{\Omega_1}\, \omega^2\,\mathbf{w}\cdot \overline{\mathbf{\varphi}}\,\mathrm{d}\mathrm\bx -2\int_{\Omega^c_1}\, \omega^2\,\mathbf{w}\cdot \overline{\mathbf{\varphi}}\,\mathrm d\mathrm\bx,\\
 \mathcal{F}(\mathbf{\phi})&:=-\int_{\Gamma}\widetilde{\mathbf{g}}\cdot\overline{\mathbf{\phi}}\,\mathrm d\sigma+\int_{\partial \Omega_2}\mathcal{T}_{\nu}(\mathbf{u_{\,\widetilde{\mathbf{f}}}})\cdot \overline{\mathbf{\phi}}\,\mathrm d\sigma-\int_{\Omega^c_1}(\mathcal{C}:\nabla \mathbf{\bu_{\,\widetilde{\bff}}}):\nabla\overline{\mathbf{\phi}}\,\mathrm d\mathrm \bx-\int_{\Omega^c_1}\omega^2\,\mathbf{\bu_{\,\widetilde{\bff}}}\,\cdot\overline{\mathbf{\phi}}\,\mathrm d\mathrm \bx.
\end{align*}
In this context, the variable $\bw$ is constructed in the same way as in Theorem 2.1 from \cite{DLM2023}, and $\mathbf{\bu_{\,\widetilde{\bff}}}$ is the unique solution to the Dirichlet boundary value problem defined by:
 \begin{equation*}
\Delta^*\mathbf{\bu_{\,\widetilde{\bff}}}+\omega^2\,\mathbf{\bu_{\,\widetilde{\bff}}}=\mathbf{0} \quad\mbox{in}\quad\Omega^c_1,\quad
\mathbf{u}_{\,\widetilde{\bff}}=\widetilde{\bff} \quad\mbox{on}\quad \Gamma,\quad
\mathbf{u}_{\,\widetilde{\bff}}={\bf 0}\quad\mbox{on}\quad {\partial \Omega}.
 \end{equation*}
It is evident that $a_1(\cdot,\cdot)$  is strictly coercive and bounded. By the Lax-Milgram lemma, there exists a bounded inverse operator $\mathcal{L}:H^1(\Omega)^n\longrightarrow H^1\Omega)^n$ such that
$$
a_1(\bw,\mathbf{\phi})=\langle\mathcal{L}\bw, \mathbf{\phi}\rangle,
$$
where $\langle\cdot,\cdot\rangle$ is the inner product in $H^1(\Omega)^n$. Furthermore, it follows that there exists a compact and bounded operator $\mathcal{K}$ satisfying $$a_1(\bw,\mathbf{\phi})=\langle\mathcal{K}\bw, \mathbf{\phi}\rangle.$$
Since $\mathcal{L}$ is bounded and $\mathcal{K}$ is compact, we conclude that $\mathcal{L}-\mathcal{K}$ is a Fredholm operator of index zero. According to the Fredholm alternative theorem, Riesz representation theory, and the uniqueness of Problem~\eqref{eq:elast1}, there exists a unique solution to Problem~\eqref{eq:elast1}. Because the inverse of $\mathcal{L}-\mathcal{K}$ is bounded,  we apply the Lax-Milgram lemma to the following equation
$$\big\langle\left(\mathcal{L}-\mathcal{K}\right)\bw,\mathbf{\phi} \big\rangle=\mathcal{F}(\mathbf{\phi}),
$$
which leads to
\begin{align*}
\|\bw\|_{H_{\partial\Omega_2}^1(\Omega)^n}&\leq\Big\|\bw\big|_{\Omega_1}\Big\|_{H^1(\Omega_1)^n}+\Big\|\bw\big|_{\Omega^c_1}\Big\|_{H^1(\Omega^c_1)^n}\leq \|\mathcal{F}\|
 \leq C\big(\|\widetilde{\bff}\|_{H^{\frac{1}{2}}(\Sigma)^n}+\|\widetilde{\bg}\|_{H^{-\frac{1}{2}}(\Sigma)^n}\big)\\
 &\leq C\big(\|\bff\|_{H_{00}^{\frac{1}{2}}(\widetilde{\Sigma})^n}+\|\bg\|_{H^{-\frac{1}{2}}_0(\widetilde{\Sigma})^n}\big).
 \end{align*}

The proof is complete.
\end{proof}

\subsection{Admissible assumptions for inverse problems}
In this subsection, we focus on the prior information regarding the fault $\Sigma$, the jumps $\bff$ and $\bg$, and the elastic tensor $\mathcal{C}(\bx)$, related to the inverse problem. As outlined in the previous section, the domain $\Omega$ can be partitioned into a finite number of Lipschitz subdomains, as specified in \eqref{Part}.

For the inverse problem, specific assumptions must be made about the geometry of the elastic solid $\Omega$. In this context, the parameters $\lambda$ and $\mu$ are defined as follows:
$$\lambda = \sum_{l=1}^N \chi_{\Omega_l}\lambda_l\quad\mbox{and}\quad \mu= \sum_{l=1}^N \chi_{\Omega_l}\mu_l, $$
where $\lambda_l$ and $\mu_l$ are real constants which fulfill the strong convexity condition \eqref{tensor2} and that \begin{equation}\label{eq:lamlayer}(\lambda_l,\,\,\mu_l)\neq (\lambda_{l+1},\,\,\mu_{l+1}) \quad\mbox{for }\quad l=1,2,...,N-1.
\end{equation}

We subsequently investigate the geometric properties of a localized region near a corner in two dimensions. Consider the polar coordinates $(r,\theta)$, where $\mathbf{x}=(x_1,x_2)^t=(r\cos\theta,r\sin\theta)^t\in \mathbb{R}^2$. We define an open sector $\mathcal{A}$, with boundaries $\Gamma^+$ and $\Gamma^-$ specified as follows:
 \begin{align}
 \mathcal{A}&=\left\{ \bx\in \mathbb{R}^2\big| \,\,{\bf 0}<\bx=(r\cos\theta,r\sin\theta)^t, \,\, \theta_{\min}<\theta<\theta_{\max},r>0 \right\}\Subset \mathbb{R}^2,\label{corner}\\
\Gamma^+&=\left\{\bx\in \mathbb{R}^2\big|\,\,{\bf0}<\bx=(r\cos\theta_{\max},r\sin\theta_{\max})^t, r>0  \right\},\nonumber\\
\Gamma^-&=\left\{\bx\in \mathbb{R}^2\big|\,\,{\bf0}<\bx=(r\cos\theta_{\min},r\sin\theta_{\min})^t, r>0  \right\},\nonumber
\end{align}
where $-\pi<\theta_{\min}<\theta_{\max}<\pi$ and $\theta_{\max}-\theta_{\min}\in (0,\pi)$. Let $B_h$ denote an open disk centered at $\bf 0$ of radius $h\in \mathbb{R}_+$. Define
\begin{equation}\label{eq:ball1}
S_h:=\mathcal{A}\cap B_h,\,\, \Gamma^\pm_h:=\Gamma^\pm\cap B_h,\,\,\Lambda_h:=\mathcal{A}\cap \partial B_h,\,\,\overline{S_h}:=\overline{\mathcal{A}}\cap \overline{B}_h,\,\,\mbox{and}\,\, \Sigma_{\Lambda_h}:=S_h\backslash S_{h/2}.
\end{equation}

\begin{figure}[ht]
\centering
\subfigure{\includegraphics[width=0.45\textwidth]{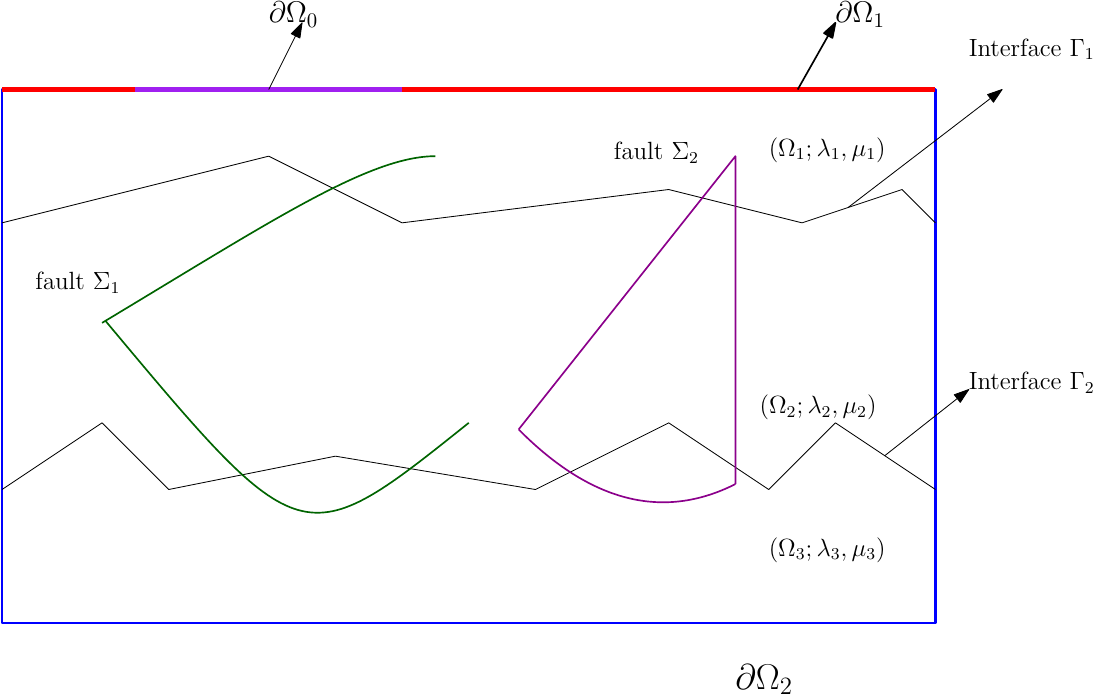}}
\subfigure{\includegraphics[width=0.45\textwidth]{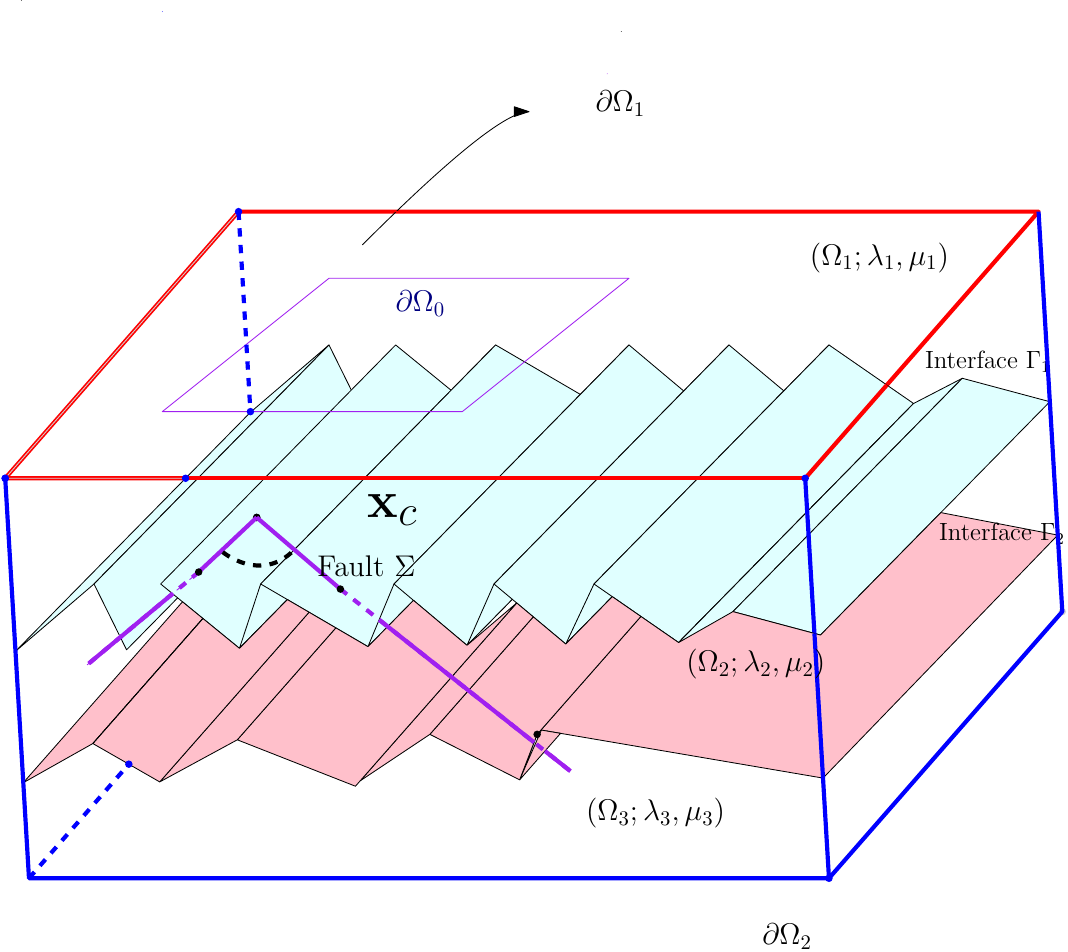}}
    \caption{Schematic illustration of some scenarios for 2D and 3D. }\label{Fig1}
\end{figure}
To recover, both locally and globally, the fault $\Sigma$ and the jumps $\mathbf{f}$ and $\mathbf{g}$ derived from a single boundary measurement from a subset $\partial\Omega_0$ of $\partial \Omega_2$, we will delineate the admissible sets $\mathcal{G}$ and $\mathcal{H}$ associated with the fault $\Sigma$, and the jumps $\mathbf{f}$ and $\mathbf{g}$, respectively.

\begin{defn}\label{def:Admis1-2}({\bf Admissible set $\mathcal{G}$})
For the bounded Lipschitz domain $\Omega = \bigcup_{l=1}^N \Omega_l$ in $\mathbb{R}^n$ (where $n=2,3$),
we say that $\Sigma$ belongs to the admissible class $\mathcal{G}$ if the following conditions are satisfied:
\begin{itemize}
\item[{\rm (i)}] In $\mathbb R^2$, $\Sigma\Subset \mathbb{R}^{2}$ is an oriented  Lipschitz curve. There exists at least one planar corner point $\bx_c$ on $\Sigma$ with the geometric properties that $\Gamma^\pm_{\bx_c,h}\Subset \Sigma$, where  $\Gamma^\pm_{\bx_c,h}=\partial \Gamma^\pm_{\bx_0}\cap B_h(\bx_c)$ and $S_{\bx_c,h}=B_h(\bx_c)\cap \mathcal{A}_{\bx_c}=B_h(\bx_c)\cap \Omega_1$, $\mathcal{A}_{\bx_c}$ and $B_h(\bx_c)$ are defined in \eqref{eq:ball1} and $\Omega_-= {\mathrm {enclose}}(\Sigma)$. Namely, if $\Sigma$ is open, then $\partial \Omega_-=\widetilde{\Sigma}=\overline{\Sigma}\cup\Sigma_0$ is a closed Lipschitz curve or surface that extends $\Sigma$ and satisfies
$\widetilde{\Sigma}\cap \partial \Omega=\emptyset$, where $\Sigma_0$ is a curve or surface connecting to the boundary $\partial\Sigma$ and ensuring that $ \Sigma_0\cap( \Sigma\backslash \partial \Sigma) =\emptyset $. If $\Sigma$ is closed, then $\partial \Omega_-=\Sigma$. Furthermore, each planar corner of $\Sigma$ must be contained within the interior of some subdomain $\Omega_l$, where $l\in\{1,2,\cdots, N\}$.
\item[{\rm (ii)}]In $\mathbb R^3$, $\Sigma\Subset \mathbb{R}^{3}$ is an oriented Lipschitz surface. We assume that $\Sigma$ possesses at least one 3D edge corner point $\bx_c=(\bx'_c,x^3_c)^t\in\mathbb{R}^3$, where $\bx'_c\in \mathbb{R}^2$ is a planar corner point. For sufficiently small positive values of $h$ and $M$, we have that
$\Gamma^\pm_{\bx'_c,h}\times (-M,M)\Subset \Sigma$ and $B'_h(\bx'_c)\times (-M,M)\cap\Omega_-=\mathcal{A}'_{\bx_c,h}\times (-M,M)$. Here, $\Gamma^\pm_{\bx'_c,h}$ denotes the two edges of a sectorial corner at $\bx'_c$, while $B'_h(\bx'_c)$  represents an open disk centered at
 $\bx'_c$ with  radius $h$, as defined in \eqref{eq:ball1}. The opening angle of the sectorial corner at $\bx'_c$ corresponds to the opening angle of the associated 3D edge corner. Furthermore, each 3D edge corner of $\Sigma$ must lie within the interior of some subdomain $\Omega_l$, where $l\in\{1,2,\cdots, N\}$.
\end{itemize}
\end{defn}

Based on Definition \ref{def:Admis1-2}, the regularity assumptions for the jumps $\bff$ and $\bg$ across the fault $\Sigma\in \mathcal{G}$ are outlined in the following definition.
\begin{defn}\label{def:Admis1-3}({\bf Admissible set $\mathcal{H}$})
We say $(\Sigma;\bff,\bg)$ belongs to the admissible class $\mathcal{H}$, if the following conditions are satisfied:
\begin{itemize}
 \item[{\rm (i)}] The fault $\Sigma$ belongs to the admissible set $\mathcal{G}$;
\item[{\rm (ii)}] In $\mathbb R^2$, for any  planar corner point $\bx_c$ on $\Sigma$, we define $\bff_{\pm}:=\bff\big|_{\Gamma^{\pm}_{\bx_c,h}}$ and $\bg_{\pm}:=\bg\big|_{\Gamma^{\pm}_{\bx_c,h}}$, where $\Gamma^{\pm}_{\bx_c,h}$ is defined in \eqref{eq:ball1}. We assume that $\bff_j \in C^{1,\alpha_j}(\Gamma^j_{\bx_c,h})^2$ and  $\bg_j\in C^{\beta_j}(\Gamma^j_{\bx_c,h})^2$, with $\alpha_j,\beta_j\in (0,1)$ and $j=+,-$.
\item[{\rm (iii)}] In $\mathbb R^3$, for any 3D edge corner point $\bx_c=(\bx'_c,x^3_c)^t$ on $\Sigma$, where $\bx'_c\in \mathbb{R}^2$ being a planar corner point, we define $\bff_{\pm}:=\bff\big|_{\Gamma'^{\pm}_{\bx'_c,h}\times (-M,M)}$ and $\bg_{\pm}=\bg\big|_{\Gamma'^{\pm}_{\bx'_c,h}\times (-M,M)}$. We assume that $\bff_j \in C^{1,\alpha_j}(\Gamma'^j_{\bx'_c,h}\times (-M,M))^3$ and  $\bg_j\in C^{\beta_j}(\Gamma'^j_{\bx'_c,h}\times (-M,M))^3$, with $\alpha_j,\beta_j\in (0,1)$ and $j=+,-$. Furthermore, $\bff_j$ and $\bg_j$ are are independent of $x_3$.
\item[{\rm(iv)}] Let  $\mathcal{V}$ denote the set of 2D corners/3D edge corners of $\Sigma$.
Either the following assumptions $\mathrm{I}$ or $\mathrm{II}$ is satisfied,
\begin{align*}
 &\qquad\quad {\it Assumption \,\,I}: \,\,\mbox{For any }\,\,\bx_c\in \mathcal{V}_0, \quad \bff_-(\bx_c)\neq\bff_+(\bx_c),\\[2mm]
&\qquad\quad  {\it Assumption \,\,II}: \mbox{If there exists a point}\, \bx_c\in \mathcal{V}\,\, \mbox{such that }  \bff_-(\bx_c)=\bff_+(\bx_c),\\[2mm]
&\quad\, \qquad\qquad \qquad \qquad \mbox{then}\,\, \nabla\bff_+(\bx_c)= \nabla\bff_-(\bx_c)=\bf 0\,\,\mbox{and}\,
 \left\{\begin{array}{ll}
	\bg_+(\bx_c)\neq \Theta_{\bx_c}\bg_-(\bx_c),\qquad\mbox{ if $n=2$},\\[2mm]
\bg^{(1,2)}_+(\bx_c)\neq \Theta_{\bx'_c}\bg^{(1,2)}_-(\bx_c),\mbox{ if $n=3$},\\[0.5mm]
\mbox{or}\quad g^3_+(\bx_c)\neq 0,\\[0.5mm]
\mbox{or} \quad g^3_-(\bx_c)\neq 0.
\end{array}\right.
\end{align*}
Here, $\bg=(\bg^{(1,2)}, g^3)^t$ in 3D, and the matrices $\Theta_{\bx_c}$ and $\Theta_{\bx'_c}$ are given by:
\beq\nonumber
\qquad\quad  \bg_\pm=(\bg^{(1,2)}_\pm,\, g^3_\pm),\,
\Theta_{\bx_c}=
		\begin{bmatrix}-\cos\theta_{\bx_c}&-\sin\theta_{\bx_c}\\
    -\sin\theta_{\bx_c} &\cos\theta_{\bx_c}
    \end{bmatrix},\,
\Theta_{\bx'_c}=\begin{bmatrix}-\cos\theta_{\bx'_c}&-\sin\theta_{\bx'_c}\\
    -\sin\theta_{\bx'_c} &\cos\theta_{\bx'_c}
    \end{bmatrix},
\eeq
where $\theta_{\bx_c}$ and $\theta_{\bx'_c}$ denote the opening angle at the 2D corner/3D edge corner  point of $\Sigma$.
\end{itemize}
\end{defn}
\begin{rem}
The admissible assumptions in (i)--(iv) are often satisfied in typical physical scenarios. Specifically, \textit{Assumption I} in (iv) is generally valid based on physical intuition when a corner exists on the fault. If \textit{Assumption I} in (iv) is violated, the jump \(\mathbf{g}\) associated with the traction field fails to exhibit continuous rotational behavior at the 2D or 3D edge corner, meaning \textit{Assumption II} in (iv) holds. This condition depends on the size of the opening angle. However, such scenarios are common and easily achievable in typical physical situations.
\end{rem}
\subsection{Main results}
This subsection presents the main uniqueness result for the inverse problem \eqref{eq:ip1}. Theorem \ref{th:main_loca} establishes a local uniqueness result for the inverse problem \eqref{eq:ip1}, which is derived from a single displacement measurement on $\partial \Omega_0$. The proof is deferred to Section \ref{proofs}.

\begin{thm}\label{th:main_loca}
Let $(\Sigma_1; \bff^1,\bg^1)$ and $(\Sigma_2; \bff^2,\bg^2)$ be elements of $\mathcal{H}$. Assume that $\mathrm{supp}(\bg^i) = \mathrm{supp}(\bff^i) = \overline{\Sigma}_i$ for $i = 1, 2$. Let $\bu_i$ denote the unique solution to Problem \eqref{eq:elast1} in $H^1_{\partial\Omega_2}(\Omega\backslash\overline{\Sigma}_i)$, with respect to $\bg = \bg^i$, $\bff = \bff^i$, and $\Sigma = \Sigma_i$, for $i = 1, 2$. Additionally, let ${\Gamma^i_1,\ldots,\Gamma^i_{N-1}} \Subset \mathcal{T}$ represent the interfaces associated with the partition of $\Omega$, defined as $\Omega = \bigcup_{l=1}^N \Omega^i_l$, corresponding to $(\Sigma_i; \bff^i,\bg^i)$ for $i = 1, 2$. If $\bu_1\big|_{\partial\Omega_0} = \bu_2\big|_{\partial\Omega_0}$, then it follows that $\Gamma^1_l \Delta \Gamma^2_l$ for $l = 1, 2, \ldots, N-1$ and $\Sigma_1 \Delta \Sigma_2 $ must not contain corners or 3D edge corners.
\end{thm}

In Theorems \ref{th:main1} and \ref{th:main2}, we establish global uniqueness results for the inverse problem \eqref{eq:ip1}, specifically concerning the determination of the interfaces $\Gamma_1,\cdots,\Gamma_{N-1}$ specified in \eqref{Part}, $\Sigma$, as well as the jumps $\mathbf{f}$ and $\mathbf{g}$ from a single boundary measurement on a subset $\partial\Omega_0$ of $\partial \Omega_2$. The proofs of Theorems \ref{th:main1} and \ref{th:main2} are deferred to Section \ref{proofs}. Next, we recall the concepts of piecewise polygonal curves and surfaces.

\begin{defn}\label{def:poly}\cite[Definition 3.2]{DLM2023}
Let $\Gamma\Subset \mathbb R^n$ (where $n=2,3$) be open. In the case of rigid motion, let $\Gamma\Subset \mathbb R^2$ be the graph of a function $f(x_1)$, where $x_1\in [a,b]$. If $[a,b]=\cup_{i=1}^\ell [a_i,a_{i+1}]$ with $\ell\geq 3$, $a_i<a_{i+1}$, $a_1=a$ and $a_\ell=b$, and if $f$ is a piecewise linear polynomial on each interval $[a_i,a_{i+1}]$, then $\Gamma\Subset \mathbb R^2$ is referred to as a piecewise polygonal curve. In the case of rigid motion, let $\Gamma\Subset \mathbb R^3$ be the graph of a function $f(x_1,x_3)$, where $(x_1,x_3)\in [a_1,a_2] \times [b_1,b_2]$. If $f(x_1,c)$ (with $c\in [b_1,b_2]$ fixed) satisfies $f(x_1,c)=g(x_1)$, where the graph of $g(x_1)$ is a piecewise polygonal curve as defined in $\mathbb R^2$, then $\Gamma\Subset \mathbb R^3$  is referred to as a piecewise polyhedral surface.
\end{defn}

\begin{thm}\label{th:main1}
  Let $(\Sigma_1; \bff^1,\bg^1)$ and $(\Sigma_2; \bff^2,\bg^2)$ be elements of $\mathcal{H}$, where $\Sigma_1$ and $\Sigma_2$ are closed. Assume that $\Omega^1_- = {\mathrm{enclose}}(\Sigma_1)$ and $\Omega^2_- = {\mathrm{enclose}}(\Sigma_2)$ are convex polygons in $\mathbb{R}^2$ or convex polyhedrons in $\mathbb{R}^3$. Here, $\Sigma_i = \bigcup_{k=1}^{m_i} \Pi_{i,k}$ denotes the union of $m_i$ edges or surfaces, where $\Pi_{i,k}$ represents the $k$-th edge or surface of the polygon or polyhedron $\Omega^i_-$ for $i = 1, 2$. Let $\bu_i$ denote the unique solution to Problem \eqref{eq:elast1} in $H^1_{\partial\Omega_2}(\Omega\backslash\overline{\Sigma}_i)$ with respect to $\bg^i$ and $\bff^i$, respectively, for $i = 1, 2$. Additionally, assume that the interfaces ${\Gamma^i_1,\ldots,\Gamma^i_{N-1}}$ of $\Omega$ that correspond to $(\Sigma_i; \bff^i,\bg^i)$ for $i = 1, 2$ are piecewise polygonal (See Definition \ref{def:poly}). If $\bu_1\big|_{\partial\Omega_0} = \bu_2\big|_{\partial\Omega_0}$, then it follows that $\Gamma^1_l = \Gamma^2_l$ for $l = 1, 2, \ldots, N-1$, and $\Sigma_1 = \Sigma_2$. This implies that $m := m_1 = m_2$ and $\Pi_k := \Pi_{1,k} = \Pi_{2,k}$ for $k = 1, \ldots, m$.

Furthermore, assume that $\bff^i$ and $\bg^i$ are piecewise constant functions on $\Pi_k$. In this case, we obtain

\begin{equation}\label{eq:unque1}
(\bff^1 - \bff^2)\big|_{\Pi_{k+1}} = (\bff^1 - \bff^2)\big|_{\Pi_{k}}, \quad (\bff^1 - \bff^2)\big|_{\Pi_1} = (\bff^1 - \bff^2)\big|_{\Pi_m}, \quad k = 1, \ldots, m - 1.
\end{equation}
and
\beq\label{eq:unque1'}
(\bg^1 - \bg^2)\big|_{\Pi_{k+1}} = \Theta_{\bx_k}(\bg^1 - \bg^2)\big|_{\Pi_k}, \quad (\bg^1 - \bg^2)\big|_{\Pi_1} = \Theta_{\bx_m}(\bg^1 - \bg^2)\big|_{\Pi_m},
\eeq
where $\Theta_{\bx_k} = \begin{bmatrix}-\cos\theta_{\bx_k} & -\sin\theta_{\bx_k} \\ -\sin\theta_{\bx_k} & \cos\theta_{\bx_k}\end{bmatrix}$ is defined similarly to $\Theta_{\bx_c}$ in Definition~\ref{def:Admis1-3}, with $\theta_{\bx_k}$ corresponding to the opening angle at $\bx_k$.
\end{thm}

\begin{thm}\label{th:main2}
Let $(\Sigma_1; \bff^1,\bg^1)$ and $(\Sigma_2; \bff^2,\bg^2)$ be elements of $\mathcal{H}$, where $\Sigma_1$ and $\Sigma_2$ are open. Suppose that $\Sigma_1$ and $\Sigma_2$ and the interfaces of $\Omega$ corresponding to $(\Sigma_i; \bff^i,\bg^i)$ for  $i = 1, 2$ are piecewise polygonal (See Definition \ref{def:poly}). Let $\bu_i$ denote the unique solution to Problem \eqref{eq:elast1} in $H^1_{\partial\Omega_2}(\Omega\backslash\overline{\Sigma}_i)$ with respect to $\bg^i$, $\bff^i$, and $\Sigma_i$, where $\bff^i \in H^{\frac{1}{2}}_{00}(\Sigma_i)$ and $\bg^i \in H^{-\frac{1}{2}}_0(\Sigma_i)$, with $\mathrm{supp}(\bg^i) = \mathrm{supp}(\bff^i) = \overline{\Sigma}_i$ for $i = 1, 2$. Furthermore, let ${\Gamma^i_1,\ldots,\Gamma^i_{N-1}} \Subset \mathcal{T}$ denote the interfaces corresponding to the partition of $\Omega$, defined as $\Omega = \bigcup_{l=1}^N \Omega^i_l$, associated with $(\Sigma_i; \bff^i,\bg^i)$ for $i = 1, 2$. If

\begin{align}\label{eq:thm33}
\bu_1\big|_{\partial\Omega_0} = \bu_2\big|_{\partial\Omega_0},
\end{align}
then
\beq\nonumber
\Gamma^1_l = \Gamma^2_l, \quad \Sigma_1 = \Sigma_2, \quad \bff^1 = \bff^2, \quad \mbox{and} \quad \bg^1 = \bg^2 \quad \mbox{for} \quad l = 1, 2, \ldots, N - 1.
\eeq
\end{thm}

\section{Auxiliary results}\label{se:pre}
\subsection{Preliminaries on CGO solutions}
In this subsection, we derive several auxiliary propositions that are essential for proving our main results in Theorems \ref{th:main_loca} through \ref{th:main2}. To accomplish this, we employ two types of complex geometrical optics (CGO) solutions. First, we introduce the CGO solution $\mathbf{u}_0$ (cf. \cite{Haher1998}), which has the following form and properties:
\begin{equation}\label{eq:cgo1}
\bu_0(\bx)=e^{\bxi\cdot(\bx-\bx_c)}\eeta,
\end{equation}
where
\begin{align}
&\kappa_{\mathrm s}:=\omega\sqrt{1/\mu},  \,\quad\qquad\qquad \kappa_{\mathrm p}:=\omega\sqrt{1/(\lambda+2\mu)}\nonumber,\\
&\bxi=\tau\bd+\mathrm{i}\sqrt{\kappa^2_\mathrm s+\tau^2}\,\bd^\perp,\quad \eeta=\bd^\perp-\mathrm{i}\sqrt{1+\kappa^2_\mathrm s/\tau^2}\,\bd,\nonumber\\
 &\tau>\kappa_\mathrm s,\qquad \bd\cdot \bd^\perp=0,\,\,\quad \bd^\perp,\, \bd\in \mathbb{S}^1,\label{eq:cgo2'}\\
 &\bd\cdot(\bx-\bx_c)\leq\,\zeta_0<0\quad \,\mbox{for}\quad \bx\in \mathcal{A}.\nonumber
\end{align}
Here, $\omega\in\mathbb{R}_+$ signifies the angular frequency of the time-harmonic elastic displacement. write $\bd=(d_1,d_2)^t$, $\bd^\perp=(d_1^\perp,d_2^\perp)^t$, $\bxi=(\xi_1,\xi_2)^t$ and $\eeta=(\eta_1,\eta_2)^t$.
By performing a series of calculations, we can verify that
\begin{align}
	&\mathcal{T}_{\nu}\bu_0= 2\mu\eeta \,e^{\bxi\cdot (\bx-\bx_c)}(\bxi\cdot \nu)+\frac{\mu\kappa^2_{\mathrm s}}{\tau}e^{\bxi\cdot(\bx-\bx_c)}(d_1\,d_2^\perp-d_2\,d_1^\perp)\nu^\perp,\label{eq:T_cgo2}\\
	&\Delta^* \,\bu_0+ \,\omega^2 \,\bu_0={\bf 0} \quad\mbox{for }\quad \bx\in \mathbb{R}^2.\label{eq:cgo2_pro1}
\end{align}
It is direct to get
\begin{align*}
&|\xi|=\sqrt{2\tau^2+\kappa^2_\mathrm s}\qquad\mbox{and}\qquad|\eta|=\sqrt{2+\kappa^2_\mathrm s/\tau^2}\leq\sqrt{3}.
\end{align*}

Several important asymptotic estimates for the final CGO solution need review. The next lemma presents asymptotic estimates for $\Re \boldsymbol{\xi}$ related to the volume integral associated with the CGO solution $\mathbf{u}_0$ near the corner point $\mathbf{x}c$ of $S_{\mathbf{x}_c,h}$, which are key components for deriving Proposition~\ref{prop:trans1}.

\begin{lem}
Let $S_{\bx_c,h}$ be defined as in \eqref{eq:ball1}. Consider the CGO solution $\bu_0$ defined in \eqref{eq:cgo1}. We have the following estimates:

\beq\label{cgo2-1}
\Big| \int_{ S_{\bx_c,h}} \bu_0\,\mathrm{dx}\Big| \leq \sqrt{3}\int_{S_{\bx_c,h}} \Big|e^{\bxi\cdot (\bx-\bx_c)}\Big|\,\mathrm{dx}\leq C'_1|\Re\bxi|^{-2}
\eeq
and
\beq
\nonumber
\Big|  \int_{S_{\bx_c,h}} e^{\bxi\cdot (\bx-\bx_c)}\,\bff\cdot \bg\,\mathrm{dx}\Big|\leq C'_2|\Re\bxi|^{-B-2/p}\,\|\bg\|_{L^q(S_{\bx_c,h})^2}
\eeq
as $|\Re \bxi| \rightarrow +\infty$, where $\bff$ and $\bg$ are vector-type functions satisfying $|\bff|\leq A|\bx-\bx_c|^B$ with $A$ and $B$ as positive constants, and $\bg\in L^q(S_{\bx_c,h})^2$ with $1/p+1/q=1$.
\end{lem}

We shall review the next lemma used for deriving curtail auxiliary propositions~\ref {prop:trans1}.
\begin{lem}\label{lem:esti3} For any given constants $\alpha>0$, $h\in(0,e)$, and $\zeta \in \mathcal{C}$ satisfying $\Re \zeta>0$. As $\Re \zeta$ goes to $\infty$, we have the following estimates
\begin{align*}
&\Big|\int^{+\infty}_{h}  r^\alpha e^{-\zeta r}     \Big|\leq \frac{2}{\Re \zeta } e^{-h \Re \zeta/2},\\
&\int_0^h r^{\alpha}e^{-\zeta r}\mathrm{d}\mathrm{r}=\frac{\Gamma(\alpha+1)}{\zeta^{\alpha+1}}+\Oh(\frac{2}{\Re \zeta}e^{-\frac{h}{2}\,\Re \zeta}),
\end{align*}
where $\Gamma(\cdot)$ is the Gamma function.
\end{lem}

\subsection{Local uniqueness results for 2D case}
To prove Theorems~\ref{th:main1} and \ref{th:main2} in $\mathbb{R}^2$, we require a crucial auxiliary proposition. Since $\Delta^*$ is invariant under rigid motions, we can assume, without loss of generality, that the corner point $\mathbf{x}_c$ is located at the origin.
\begin{prop}\label{prop:trans1}
Let $S_h$ and $\Gamma^\pm_h$ be defined in \eqref{eq:ball1}. Let $\bv\in H^1(S_h)^2$ and $\bw\in H^1(S_h)^2$ such that
\begin{equation*}
\begin{cases}
\Delta^*\,\bv+\omega^2\,\bv=\mathbf{0},\,\,\Delta^*\,\bw+\omega^2\,\bw=\mathbf{0}  &\mbox{in}\quad S_h,\\
  \mathcal{T}_{\nu}\,\bw-\mathcal{T}_{\nu}\,\bv=\bg_1,\quad \,\,\bw-\bv=\bff_1 &\mbox{on}\quad \Gamma^+_h,\\
\mathcal{T}_{\nu}\,\bw-\mathcal{T}_{\nu}\,\bv=\bg_2,\quad\,\, \bw-\bv=\bff_2 &\mbox{on}\quad \Gamma^-_h,\\
\end{cases}
\end{equation*}
 with $\nu\in \mathbb{S}^1$ denoting the exterior unit normal vector to $\Gamma^\pm_h$, $\bff_1\in H^{\frac{1}{2}}(\Gamma_h^+)^2\cap C^{1,\alpha_1}(\Gamma_h^+)^2$, $\bff_2\in H^{\frac{1}{2}}(\Gamma_h^-)^2\cap C^{1,\alpha_2}(\Gamma_h)^2$, $\bg_1\in H^{-\frac{1}{2}}(\Gamma_h^+)^2\cap C^{\beta_1}(\Gamma_h^+)^2 $ and $\bg_2\in H^{-\frac{1}{2}}(\Gamma_h^-)^2\cap C^{\beta_2}(\Gamma_h^-)^2 $, where $\alpha_1,\alpha_2,\beta_1 $ and $\beta_2 $ are in $(0,1)$.
 Recall that the CGO solution $\bu_0(\bx)$ is defined in \eqref{eq:cgo1} with an asymptotic parameter $\tau\in \mathbb{R}_+$. Then, we have
\begin{align}
 &\Re\bg_1({\bf 0})\cdot\int_{\Gamma^+_h}\bu_0\,\mathrm d\sigma+\bg_2({\bf 0})\cdot\int_{\Gamma^-_h}\bu_0\,\mathrm d\sigma-\bff_1({\bf 0})\cdot\int_{\Gamma^+_h}\mathcal{T}_{\nu}\,\bu_0\,\mathrm d\sigma-\bff_2({\bf 0})\cdot\int_{\Gamma^-_h}\mathcal{T}_{\nu}\,\bu_0\,\mathrm d\sigma\nonumber\\
&=\int_{\Gamma^+_h}\mathcal{T}_{\nu}\bu_0\cdot\delta\bff_1(\bx)\,\mathrm d\sigma-\int_{\Gamma^+_h}\delta\bg_1\cdot\bu_0\,\mathrm d\sigma+\int_{\Gamma^-_h}\mathcal{T}_{\nu}\bu_0\cdot\delta\bff_2(\bx)\,\mathrm d\sigma-\int_{\Gamma^-_h}\delta\bg_2\cdot\bu_0\,\mathrm d\sigma\nonumber\\
&\,\,\,\,+\int_{\Lambda_h}\mathcal{T}_{\nu}\big(\bv-\bw\big)\cdot\bu_0-\mathcal{T}_{\nu}\bu_0\cdot\big(\bv-\bw\big)\,\mathrm d\sigma\label{eq:inter0}.
 \end{align}
 and
 \beq\label{eq:local2}
 \bff_1({\bf 0})\big|_{\Gamma^+_h}=\bff_2({\bf 0})\big|_{\Gamma^-_h}.
 \eeq
 Furthermore, if the condition that $\nabla \bff_1({\bf 0})=\nabla \bff_2({\bf 0})={\bf 0}$ holds, we obtain
 \begin{align}\label{eq:local1}
\bg_1({\bf 0})\big|_{\Gamma^+_h}&=\Theta\,\bg_2({\bf 0})\big|_{\Gamma^-_h},
\end{align}
where $\Theta=\begin{bmatrix}-\cos(\theta_{\max}-\theta_{\min}),&-\sin(\theta_{\max}-\theta_{\min})\\
    -\sin(\theta_{\max}-\theta_{\min}), &+\cos(\theta_{\max}-\theta_{\min})
    \end{bmatrix}.$
\end{prop}
\begin{rem}\nonumber
It is essential to note that our approach accommodates both complex-valued and real-valued functions for $\mathbf{u}(\mathbf{x})$. In this paper, we specifically consider $\mathbf{u}(\mathbf{x})$ to be complex-valued to emphasize the generality of the corresponding method.
\end{rem}

\begin{proof}[Proof of Proposition~\ref{prop:trans1}]
Thanks to the symmetrical role of $(\Re\bv,\Re\bw)$ and $(\Im\bv,\Im\bw)$. We only need to check the relevant results for $(\Re\bv,\Re\bw)$. By a similar proof, we can prove that these results are still valid for $(\Im\bv,\Im\bw)$, so for $(\bv,\bw)$. From Betti's second formula and \eqref{eq:cgo2_pro1}, we have
 \begin{align}
&\int_{\Gamma^+_h}\Re\bg_1\cdot\bu_0-\mathcal{T}_{\nu}\bu_0\cdot\Re\bff_1\,\mathrm d\sigma+\int_{\Gamma^-_h}\Re\bg_2\cdot\bu_0-\mathcal{T}_{\nu}\bu_0\cdot\Re\bff_2\,\mathrm d\sigma\nonumber\\
&=\int_{\Lambda_h}\mathcal{T}_{\nu}\big(\Re\bv-\Re\bw\big)\cdot\bu_0-\mathcal{T}_{\nu}\bu_0\cdot\big(\Re\bv-\Re\bw\big)\, \mathrm d\sigma. \label{iden:1}
\end{align}
Given $\bff_1\in C^{\alpha_1}(\Gamma_h^+)^2$, $\bff_2\in C^{\alpha_2}(\Gamma_h)^2$, $\bg_1\in C^{\beta_1}(\Gamma_h^+)^2$, and $\bg_2\in C^{\beta_2}(\Gamma_h^-)^2$, we can express them using expansions as follows:
\begin{equation}\label{eq:expan}
\begin{cases}
 \bff_1(\bx)=\bff_1({\bf0})+ \nabla\bff_1({\bf 0})\cdot\bx+\delta\bff_1(\bx),\,\,\quad \big| \delta\bff_1(\bx) \big|\leq \|\bff_1\|_{C^{\alpha_1}}|\bx|^{\alpha_1+1}, \\
\bff_2(\bx)=\bff_2({\bf0})+\nabla\bff_2({\bf 0})\cdot\bx+ \delta\bff_2(\bx),\,\,\quad \big| \delta\bff_2 (\bx)\big|\leq \|\bff_2\|_{C^{\alpha_2}}|\bx|^{\alpha_2+1},  \\
 \bg_1(\bx)=\bg_1({\bf0})+ \delta\bg_1(\bx),\hspace{62pt}  \big| \delta\bg_1 (\bx)\big|\leq \|\bg_1\|_{C^{\beta_1}}|\bx|^{\beta_1},  \\
 \bg_2(\bx)=\bg_2({\bf0})+ \delta\bg_2(\bx),\hspace{62pt}  \big| \delta\bg_2(\bx) \big|\leq \|\bg_2\|_{C^{\beta_2}}|\bx|^{\beta_2}.
 \end{cases}
 \end{equation}
By applying these expansions, we can rewrite \eqref{iden:1} as
 \begin{align}
 &\Re\bg_1({\bf 0})\cdot\int_{\Gamma^+_h}\bu_0\,\mathrm d\sigma+\Re\bg_2({\bf 0})\cdot\int_{\Gamma^-_h}\bu_0\,\mathrm d\sigma-\Re\bff_1({\bf 0})\cdot\int_{\Gamma^+_h}\mathcal{T}_{\nu}\,\bu_0\,\mathrm d\sigma-\Re\bff_2({\bf 0})\cdot\int_{\Gamma^-_h}\mathcal{T}_{\nu}\,\bu_0\,\mathrm d\sigma\nonumber\\
 &=\int_{\Gamma^+_h}\mathcal{T}_{\nu}\bu_0\cdot\Big(\Re\nabla\bff_1({\bf 0})\cdot \bx+\Re\delta\bff_1\Big)\,\mathrm d\sigma-\int_{\Gamma^+_h}\Re\delta\bg_1\cdot\bu_0\,\mathrm d\sigma+\int_{\Gamma^-_h}\mathcal{T}_{\nu}\bu_0\cdot\Big(\Re\nabla\bff_2({\bf 0})\cdot \bx+\Re\delta\bff_2\Big)\,\mathrm d\sigma\nonumber\\
 &\,\,\,-\int_{\Gamma^-_h}\Re\delta\bg_2\cdot\bu_0\,\mathrm d\sigma+\int_{\Lambda_h}\mathcal{T}_{\nu}\big(\Re\bv-\Re\bw\big)\cdot\bu_0-\mathcal{T}_{\nu}\bu_0\cdot\big(\Re\bv-\Re\bw\big)\,\mathrm d\sigma\label{eq:inter1}.
 \end{align}
By replacing $(\Re\bv,\Re\bw)$ with $(\Im\bv,\Im\bw)$,we obtain a similar integral identity, implying that \eqref{eq:inter0} remains valid. Additionally, from \eqref{eq:T_cgo2}, we can deduce the following expression
\begin{align}\label{eq:esti2_1}
\int_{\Gamma^+_h}  \mathcal{T}_{\nu_{\max}}\bu_0  \,\mathrm d\sigma&=\int_{\Gamma_h^+}\bigg[ 2\mu\eeta(\bxi\cdot\nu_{\max}) e^{\bxi\cdot r\hat{\bx}_{\max}} +\frac{\mu\kappa_{\mathrm s}^2}{\tau}e^{\eeta\cdot \bx}\nu^{\perp}_{\max}\bigg]\mathrm{d}\sigma\nonumber\\
&=\Big[2\mu\eeta(\bxi\cdot\nu_{\max})+ \frac{\mu\kappa_{\mathrm s}^2}{\tau}\nu^{\perp}_{\max} \Big]\int_0^he^{\bxi\cdot r\hat{\bx}_{\max}}\mathrm{d}r\nonumber\\
&=\Big[ 2\mu\eeta(\bxi\cdot\nu_{\max})+ \frac{\mu\kappa_{\mathrm s}^2}{\tau}\nu^{\perp}_{\max} \Big]\Big(\frac{1}{-\bxi\cdot \hat{\bx}_{\max}}-\int_h^\infty e^{\bxi\cdot r\hat{\bx}_{\max}}\mathrm{d}r\Big)
\end{align}
By employing similar arguments, we can derive the following results
\begin{align}
\int_{\Gamma^-_h}  \mathcal{T}_{\nu_{\min}}\bu_0  \,\mathrm {d}\sigma&=
\Big[ 2\mu\eeta(\bxi\cdot\nu_{\min})+ \frac{\mu\kappa_{\mathrm s}^2}{\tau}\nu^{\perp}_{\min} \Big]\Big(\frac{1}{-\bxi\cdot \hat{\bx}_{\min}}-\int_h^\infty e^{\bxi\cdot r\hat{\bx}_{\min}}\,\mathrm{d}r\Big),\label{eq:esti2_2}\\
\int_{\Gamma^+_h}\bu_0\,\mathrm{d}\sigma&=\frac{\eeta}{-\bxi\cdot\hat{\bx}_{\max}}-\eeta\int_h^\infty e^{\bxi\cdot r\hat{\bx}_{\max}} \,\mathrm{d}r,\label{eq:esti2_3}\\
\int_{\Gamma^-_h}\bu_0\,\mathrm{d}\sigma&=\frac{\eeta}{-\bxi\cdot\hat{\bx}_{\min}}-\eeta\int_h^\infty e^{\bxi\cdot r\hat{\bx}_{\min}}\, \mathrm{d}r.\label{eq:esti2_4}
\end{align}
Substituting \eqref{eq:esti2_1}, \eqref{eq:esti2_2}, \eqref{eq:esti2_3}, and \eqref{eq:esti2_4} into \eqref{eq:inter1}, we obtain the following integral identity
 \begin{align}\label{eq:iden1}
&\bigg(\frac{\Re\bg_1({\bf 0})\cdot\eeta }{-\bxi\cdot \hat{\bx}_{\max}}+\frac{\Re\bg_2({\bf 0})\cdot\eeta }{-\bxi\cdot \hat{\bx}_{\min}}\bigg)+2\mu\bigg(\Re\bff_1({\bf 0})\cdot\,\frac{\eeta\,\bxi\cdot\nu_{\max}}{\bxi\cdot\nu_{\max}} +\Re\bff_2({\bf 0})\cdot\frac{\eeta\,\bxi\cdot\nu_{\min}}{\bxi\cdot\nu_{\min}}        \bigg)=\sum\limits_{j=1}^{14}R_j,
\end{align}
\allowdisplaybreaks
where
\begin{align*}
&R_1= -\Re\bff_1({\bf 0})\cdot 2\mu\,\eeta(\bxi\cdot\nu_{\max})\int_h^\infty e^{\bxi\cdot r\hat{\bx}_{\max}} \mathrm{d}r, \hspace{45pt}
 R_2= \Re\bg_1({\bf 0})\cdot\eeta\int_h^\infty e^{\bxi\cdot r\hat{\bx}_{\max}} \mathrm{d}r,\\
 &R_3= -\Re\bff_2({\bf 0})\cdot 2\mu\,\eeta(\bxi\cdot\nu_{\min})\int_h^\infty e^{\bxi\cdot r\hat{\bx}_{\min}} \mathrm{d}r,\hspace{48pt}
 R_4= \Re\bg_2({\bf 0})\cdot\eeta\int_h^\infty e^{\bxi\cdot r\hat{\bx}_{\min}} \mathrm{d}r,\\
&R_5=-\Re\bff_1({\bf 0})\cdot\frac{\mu\kappa_{\mathrm s}^2}{\tau}{\nu}^\perp_{\max}\int_h^\infty e^{\bxi\cdot r\hat{\bx}_{\max}} \mathrm{d}r,
\hspace{68pt}
R_6=\int_{\Gamma_h^+}\Re\delta\bff_1\cdot\mathcal{T}_{\nu_{\max}}\bu_0\mathrm{d}\sigma,\\
&R_7=-\Re\bff_2({\bf 0})\cdot\frac{\mu\kappa_{\mathrm s}^2}{\tau}{\nu}^\perp_{\max}\int_h^\infty e^{\bxi\cdot r\hat{\bx}_{\min}} \mathrm{d}r,
\hspace{69pt}
R_8=\int_{\Gamma_h^-}\Re\delta\bff_2\cdot\mathcal{T}_{\nu_{\min}}\bu_0\mathrm{d}\sigma,\\
&R_{9}=\int_{\Lambda_h}\mathcal{T}_{\nu}\bu_0\cdot\big(\Re\bv-\Re\bw\big)-\mathcal{T}_{\nu}\big(\Re\bv-\Re\bw\big)\cdot\bu_0\,\mathrm d\sigma,\hspace{12pt}
 R_{10}=-\int_{\Gamma_h^-}\Re\delta\bg_2\cdot\bu_0\mathrm{d}\sigma,\\
&R_{11}=\Re\bff_1({\bf 0})\cdot \frac{\mu\kappa_{\mathrm s}^2}{\tau}\,\frac{\nu_{\max}^\perp}{\bxi\cdot\hat{\bx}_{\max}}+\Re\bff_2({\bf 0})\cdot \frac{\mu\kappa_{\mathrm s}^2}{\tau}\,\frac{\nu_{\min}^\perp}{\bxi\cdot\hat{\bx}_{\min}},\hspace{11pt}
R_{12}=-\int_{\Gamma_h^+}\Re\delta\bg_1\cdot\bu_0\mathrm{d}\sigma,\\
&R_{13}=\int_{\Gamma_h^-}\big(\nabla\bff_2({\bf 0})\cdot\bx\big) \cdot\mathcal{T}_{\nu_{\min}}\bu_0\mathrm{d}\sigma,
\hspace{86pt} R_{14}=\int_{\Gamma_h^+}\big(\nabla\bff_2({\bf 0})\cdot\bx\big) \cdot\mathcal{T}_{\nu_{\max}}\bu_0\mathrm{d}\sigma.\\
\end{align*}
From the fact that $|\xi|=\sqrt{2\tau^2+\kappa^2_\mathrm s}$ and $|\eta|=\sqrt{2+\kappa^2_\mathrm s/\tau^2}\leq\sqrt{3}$, along with the expressions for $R_1$--$R_5$, $R_7$ and Lemma~\ref{lem:esti3}, we can derive the following estimates for sufficiently large $\tau$,
\begin{alignat}{4}
&\big|R_1\big|=\Oh(\tau^{-1}e^{-c_1\tau}),&\quad\big|R_2\big|&=\Oh(e^{-c_2\,\tau}), &\quad\big|R_3\big|&=\Oh(\tau^{-1}e^{-c_3\tau}),\nonumber\\
&\big|R_4\big|=\Oh(e^{-c_4\,\tau}),&\quad\big|R_5\big|&=\Oh(\tau^{-1}e^{-c_5\,\tau}),&\quad\big|R_7\big|&=\Oh(\tau^{-1}e^{-c_7\,\tau}),\label{eq:R1}
\end{alignat}
where $c_1,\,c_2,\, c_3,\,c_4,\,c_5$ and $c_7$ are all positive constants independent of $\tau$.
Next, combining the estimates of $\delta\bff_1$ in \eqref{eq:expan}, the expression for $\mathcal{T}_{\nu_{\max}}\bu_0$ in \eqref{eq:T_cgo2}, and Lemma~\ref{lem:esti3}, we find
\begin{equation}\label{inter6}
\big|R_6\big|\leq\,c'_0\,\int_0^h \tau r^{\alpha_1+1}e^{\tau\bd\cdot \bx}+\tau^{-1} r^{\alpha_1+1}e^{\tau\bd\cdot \bx}\,\mathrm{d}r= \Oh(\tau^{-\alpha_1-1} ).
\end{equation}
Similarly, we obtain the following estimates
{\small \begin{align}\label{eq:R2}
\big|R_8\big|= \Oh(\tau^{-\alpha_2-1} ),\,\,\big|R_{10}\big|= \Oh(\tau^{-\beta_1-1} ),\,\,\big|R_{12}\big|= \Oh(\tau^{-\beta_2-1} ),\, \,|R_{13}|= |R_{14}|=\Oh(\tau^{-1} ).
\end{align}}
Since the unit vector $\bd $ in $\bxi$ satisfies $\bd\cdot\bx\leq\,\zeta_0<0$ for $ \bx\in \mathcal{A}$,
we have
\beq\nonumber
\bigg|\frac{1}{\bxi\cdot\bx_{\max}}\bigg|\leq\frac{1}{
\tau|\bd\cdot\bx_{\max}|}\leq \tau^{-1}|\zeta_0|^{-1}\quad \mbox{and}\quad \bigg|\frac{1}{\bxi\cdot\bx_{\min}}\bigg|\leq\frac{1}{
\tau|\bd\cdot\bx_{\min}|}\leq \tau^{-1}|\zeta_0|^{-1}.
\eeq
This implies that
\beq\label{inter11}
\quad \big|R_{11}\big|= \Oh(\tau^{-2} )\quad \mbox{and}\quad \bigg|\frac{\Re\bg_1({\bf 0})\cdot\eeta }{\bxi\cdot \hat{\bx}_{\max}}+\frac{\Re\bg_2({\bf 0})\cdot\eeta }{\bxi\cdot \hat{\bx}_{\min}}\bigg|=\Oh(\tau^{-1}).
\eeq

From the Cauchy-Schwarz inequality and the trace theorem, we obtain
\begin{align}\label{inter9}
\big|R_{9}  \big|&\leq \|\bu_0\|_{L^2(\Lambda_h)^2}\|\mathcal{T}_{\nu}(\Re\bv-\Re\bw)\|_{L^2(\Lambda_h)^2}+\|\Re\bv-\Re\bw\|_{L^2(\Lambda_h)^2}\|\mathcal{T}_{\nu} \bu_0\|_{L^2(\Lambda_h)^2}\nonumber\\[1mm]
&\leq \big( \|\bu_0\|_{L^2(\Lambda_h)^2}+\|\mathcal{T}_{\nu}\bu_0\|_{L^2(\Lambda_h)^2}\big)\big\|\bv-\bw\big\|_{H^1(S_h)^2}\nonumber\\[1mm]
&\leq c_{9}\big\|\bu_0\big\|_{H^1(S_h)^2}\leq c'_{9} e^{-\tau\zeta_0\,h}.
\end{align}
Together with the estimates for $R_1-R_{14}$ provided in \eqref{eq:R1}, \eqref{inter6}, \eqref{eq:R2}, \eqref{inter11}, and \eqref{inter9}, the following equation results from \eqref{eq:iden1}:
\beq\label{eq:R4}
\lim\limits_{\tau\rightarrow +\infty}\Big(\Re\bff_1({\bf 0})\cdot\,\frac{\eeta\,\bxi\cdot\nu_{\max}}{\bxi\cdot\nu_{\max}} +\Re\bff_2({\bf 0})\cdot\frac{\eeta\,\bxi\cdot\nu_{\min }}{\bxi\cdot\nu_{\min}}\Big)=0.
\eeq
Moreover, since $\nabla\bff_1({\bf 0})=\nabla\bff_2({\bf 0})={\bf 0}$ and $\alpha_1,\,\alpha_2\in (0,1)$, substituting \eqref{eq:R4} into \eqref{eq:iden1} and multiplying the new identity $\tau$ yields
\beq\label{eq:R5}
\lim\limits_{\tau\rightarrow +\infty}\eeta\cdot\Big(\frac{\tau\,\Re\bg_1({\bf 0})}{\bxi\cdot\hat{\bx}_{\max}}+\frac{\tau\,\Re\bg_2({\bf 0})}{\bxi\cdot\hat{\bx}_{\min}}\Big)=0.
\eeq
\allowdisplaybreaks
Next, recalling the expressions for $\bxi$ and $\eeta$ in \eqref{eq:cgo2'},we find that \eqref{eq:R4} can be reduced to the following two equations
\begin{align}\label{eq:R6}
&\bd^\perp\cdot\Re\bff_1({\bf 0})\Big[(\bd\cdot\nu_{\max})(\bd\cdot\hat{\bx}_{\min})-(\bd^\perp\cdot\nu_{\max})(\bd^\perp\cdot\hat{\bx}_{\min})
\Big]\nonumber\\
&+\bd\cdot\Re\bff_1({\bf 0})\Big[ (\bd^\perp\cdot\nu_{\max})(\bd\cdot\hat{\bx}_{\min})+(\bd\cdot\nu_{\max})(\bd^\perp\cdot\hat{\bx}_{\min}) \Big]\nonumber\\
&+\bd^\perp\cdot\Re\bff_2({\bf 0}) \Big[(\bd\cdot\nu_{\min})(\bd\cdot\hat{\bx}_{\max})-(\bd^\perp\cdot\nu_{\min})(\bd^\perp\cdot\hat{\bx}_{\max})  \Big]\nonumber\\
&+\bd\cdot\Re\bff_2({\bf 0})\Big[(\bd^\perp\cdot\nu_{\min})(\bd\cdot\hat{\bx}_{\max})+  (\bd\cdot\nu_{\min})(\bd^\perp\cdot\hat{\bx}_{\max}) \Big]
   =0
\end{align}
and
\begin{align}\label{eq:R7}
&\bd^\perp\cdot\Re\bff_1({\bf 0})\Big[(\bd^\perp\cdot\nu_{\max})(\bd\cdot\hat{\bx}_{\min})+(\bd\cdot\nu_{\max})(\bd^\perp\cdot\hat{\bx}_{\min})
\Big]\nonumber\\
&-\bd\cdot\Re\bff_1({\bf 0})\Big[ (\bd\cdot\nu_{\max})(\bd\cdot\hat{\bx}_{\min})-(\bd^\perp\cdot\nu_{\max})(\bd^\perp\cdot\hat{\bx}_{\min}) \Big]\nonumber\\
&+\bd^\perp\cdot\Re\bff_2({\bf 0}) \Big[(\bd^\perp\cdot\nu_{\min})(\bd\cdot\hat{\bx}_{\max})+(\bd\cdot\nu_{\min})(\bd^\perp\cdot\hat{\bx}_{\max})  \Big]\nonumber\\
&-\bd\cdot\Re\bff_2({\bf 0})\Big[(\bd\cdot\nu_{\min})(\bd\cdot\hat{\bx}_{\max})- (\bd^\perp\cdot\nu_{\min})(\bd^\perp\cdot\hat{\bx}_{\max}) \Big]
  =0.
\end{align}
Fix $\theta_{\min}=0$, $\theta_{\max}=\theta$ and $\theta_0={\mathrm{arg}}(\bd)$, where $\mathrm{arg}(\cdot)$ represents the argument of the vector $\bd$. We have
\begin{align}\label{eq:cordi1}
&\hat{\bx}_{\min}=(1,0)^t, \quad \hat{\bx}_{\max}=(\cos\theta,\sin\theta)^t,\quad\quad  \nu_{\min}=(0,-1)^t,\quad\quad \nu_{\max}=(-\sin\theta,\cos\theta)^t,\nonumber\\
& \nu^{\perp}_{\min}=(1,0)^t,\quad \nu_{\max}^\perp=(-\cos\theta,-\sin\theta)^t,\,\,
\bd=(\cos\theta_0,\sin\theta_0)^t,\,\,\bd^\perp=(-\sin\theta_0,\cos\theta_0)^t.
\end{align}
With some simplifications, we can rewrite \eqref{eq:R6} and \eqref{eq:R7} as follows
\beq\label{eq:R8}
\begin{bmatrix}
\sin(2\theta_0-\theta),&\cos(2\theta_0-\theta)\\[2mm]
\cos(2\theta_0-\theta),&-\sin(2\theta_0-\theta)
\end{bmatrix}\bx={\bf 0},
\eeq
where
\begin{align*}
\bx=\begin{bmatrix} x_1\\x_2 \end{bmatrix}=\begin{bmatrix}\bd^\perp\cdot\Re\bff_1({\bf 0})-\bd^\perp\cdot\Re\bff_2({\bf 0})\\\bd\cdot\Re\bff_1({\bf 0})-\bd\cdot\Re\bff_2({\bf 0})
  \end{bmatrix}.
\end{align*}
Note that
$$\mathrm{\det}\left(\begin{bmatrix}
\sin(2\theta_0-\theta),&\cos(2\theta_0-\theta)\\[2mm]
\cos(2\theta_0-\theta),&-\sin(2\theta_0-\theta)
\end{bmatrix}\right)=-1\neq 0.$$
Thus, there exists a unique zero solution to \eqref{eq:R8}, implying that \eqref{eq:local2} holds.

From Equation \eqref{eq:R5}, we can derive the following equations
\beq\nonumber
(\bd^\perp\cdot \bg_1({\bf 0}) )(\bd\cdot\hat{\bx}_{\min})+(\bd\cdot\bg_1({\bf 0}))(\bd^\perp\cdot\hat{\bx}_{\min})+(\bd^\perp\cdot \bg_2({\bf 0})) (\bd\cdot\hat{\bx}_{\max})+(\bd\cdot\bg_2({\bf 0}))(\bd^\perp\cdot\hat{\bx}_{\max})=0
\eeq
and
\beq\nonumber
-(\bd\cdot \bg_1({\bf 0}) )(\bd\cdot\hat{\bx}_{\min})+(\bd^\perp\cdot\bg_1({\bf 0}))(\bd^\perp\cdot\hat{\bx}_{\min})-(\bd\cdot \bg_2({\bf 0})) (\bd\cdot\hat{\bx}_{\max})+(\bd^\perp\cdot\bg_2({\bf 0}))(\bd^\perp\cdot\hat{\bx}_{\max})=0.
\eeq
Based on Equation \eqref{eq:cordi1} and complex simplification,  the equations above can be equivalently recast as follows
\beq\nonumber
\Re\bg_1({\bf 0}) =\Theta\Re\bg_2({\bf 0}),\quad \mbox{where}\quad  \Theta=\begin{bmatrix}-\cos\theta,&-\sin\theta\\
    -\sin\theta, &\,\cos\theta
    \end{bmatrix}\,\,\mbox{and}\,\, \det(\Theta)\neq 0,
\eeq
which implies that Equation \eqref{eq:local1} holds.

The proof is complete.
 \end{proof}

\subsection{Local uniqueness results for 3D case}
In this subsection, we address the 3D case and introduce a dimension reduction operator $\mathcal{P}$ to analyze the continuity and rotational continuity of the jumps $\mathbf{f}$ and $\mathbf{g}$ at a 3D edge corner. We assume that the fault $\Sigma$ is a Lipschitz surface containing at least one 3D edge corner $\mathbf{x}_c = (\mathbf{x}'_c, x^3_c)^t \in \Sigma \Subset \mathbb{R}^3$. The following definition introduces the dimension reduction operator, which is essential for deriving an auxiliary proposition analogous to Proposition~\ref{prop:trans2} at a 3D edge corner.
\begin{defn}\nonumber
Let  $S_{\bx_c,h}$ be defined as in Equation \eqref{eq:ball1}, with the vertex $\bx'_c$. Consider a function $\phi$ defined over the domain $S_{\bx'_c,h}\times (-M,M)$, where $M>0$. For any fixed $x^3_c\in (-M,M)$, we assume that $\phi\in C^\infty_0(x^3_c-L,x^3_c+L)$ is a nonnegative function such that $\phi \not\equiv\emptyset$. Let $\bx=(\bx',x_3)\in \mathbb{R}^3$. The dimension reduction operator $\mathcal{P}$ is defined as follows
\begin{equation}\nonumber
\mathcal{P}(\bh)(\bx')=\int_{x^3_c-L}^{x_c^3+L}\phi(x_3)\bh(\bx',x_3)\mathrm{d}x_3,
\end{equation}
where $S_{\bx'_c,h}$ is defined in \eqref{eq:ball1} with the vertex $\bx_c'$.
\end{defn}
Before deriving the main results of this subsection, we review several key properties of this operator.
\begin{lem}\nonumber\cite[Lemma~3.1]{DLS2021}
Let $\bh\in H^m(S_{\bx'_c,h}\times (-M,M))^3$, $m=1,2$. Then
\beq\nonumber
\mathcal{P}(\bh)(\bx')\in H^m(S_{\bx'_c,h}))^3.
\eeq
Similarly, if $\bh\in C^\delta\Big(\overline{S_{\bx'_c,h}}\times [-M,M]\Big)^3$ with $\delta\in (0,1)$, then
\beq\nonumber
\mathcal{P}(\bh)(\bx')\in C^\delta\Big(\overline{S_{\bx'_c,h}}\Big)^3.
\eeq
\end{lem}

Noting that the three-dimensional isotropic elastic operator $\Delta^*$ defined in \eqref{def:Lam} can be rewritten as
\begin{align*}
\Delta^*&=\begin{bmatrix}
\lambda\Delta +(\lambda+\mu)\partial_1^2&(\lambda+\mu)\partial_1\partial_2&(\lambda+\mu)\partial_1\partial_3\\
(\lambda+\mu)\partial_1\partial_2&\lambda\Delta +(\lambda+\mu)\partial_2^2&(\lambda+\mu)\partial_2\partial_3\\
(\lambda+\mu)\partial_1\partial_3&(\lambda+\mu)\partial_2\partial_3&\lambda\Delta +(\lambda+\mu)\partial_3^2
\end{bmatrix}\\[5mm]
&=\widetilde{\Delta}^*+
\begin{bmatrix}
\lambda\partial^2_3&0&(\lambda+\mu)\partial_1\partial_3\\
0&\lambda\partial_3^2&(\lambda+\mu)\partial_2\partial_3\\
(\lambda+\mu)\partial_1\partial_3&(\lambda+\mu)\partial_2\partial_3&\lambda\partial_3^2+(\lambda+\mu)\partial_3^2
\end{bmatrix},
\end{align*}
where
\beq\label{eq:L}
\widetilde{\Delta}^*=\begin{bmatrix}
\lambda\Delta'+(\lambda+\mu)\partial_1^2&(\lambda+\mu)\partial_1\partial_2&0\\
(\lambda+\mu)\partial_1\partial_2&\lambda\Delta'+(\lambda+\mu)\partial_2^2&0\\
0&0&\lambda\Delta'
\end{bmatrix}=\begin{bmatrix}
\hat{\Delta}^*&0\\
0&\lambda\Delta'
\end{bmatrix}
\eeq
with $\Delta'=\partial_1^2+\partial_2^2$ being the Laplace operator with respect to the $\bx'$-variables. Here, the operator $\hat{\Delta}^{*}$ is the two-dimensional isotropic elastic operator with respect to the $\bx'$-variables.

For any fixed $x_c^3\in (-M, M)$ and sufficiently small $L>0$ such that $(x_c^3-L,x_c^3+L)\Subset (-M,M)$, additionally, we have $\Delta^*=\widetilde{\Delta}^*$. Since $\Delta^*$ is invariant under rigid motion, we can assume that $\bx'_c=\mathbf{ 0}$ in $\mathbb R^2.$ Thus, let $\mathcal{A}_h$ and $\Gamma^\pm_{h}$ be defined as in \eqref{eq:ball1}, with $\bx'_c$ coinciding with the origin in the 2D case. After some tedious calculations, we obtain the following lemma.
\begin{lem}\label{lem:R3_1}
Under the setup regarding $S_{h}$ and a 3D edge corner as described, suppose that $\bff_1\in C^{1,{\alpha_1}}\big(\Gamma^+_{h}\times (-M,M)\big)^3\cap H^{1/2}\big(\Gamma^+_{h}\times (-M,M)\big)^3$, $\bff_2\in C^{1,{\alpha_2}}\big(\Gamma^+_{h}\times (-M,M)\big)^3\cap H^{1/2}\big(\Gamma^+_{h}\times (-M,M)\big)^3$, $\bg_1\in C^{\beta_1}\big(\Gamma^+_{h}\times (-M,M)\big)^3\cap L^2\big(\Gamma^\pm_{h}\times (-M,M)\big)^3$, and $\bg_2\in C^{\beta_2}\big(\Gamma^-_{h}\times (-M,M)\big)^3\cap L^2\big(\Gamma^-_{h}\times (-M,M)\big)^3$, without depending on $x_3$, where $\alpha_1,\,\alpha_2,\,\beta_1$ and $\beta_2$ belong to $ (0,1)$. For $i=1,2$, write
\allowdisplaybreaks
\begin{align}\label{eq:DEC}
\bv=(\bv^{(1,2)}, \, v_3)^t,\quad \bw=(\bw^{(1,2)},\,w_3)^t,\quad \bff_i=(\bff^{(1,2)}_i,\quad f_i^{3})^t,\quad\mbox{and}\quad \bg_i=(\bg^{(1,2)}_i,\, g^{3}_i)^t.
\end{align}
Then  the following transmission eigenvalue problem for $(\bv,\bw)\in H^1\big(S_h\times (-M,M)\big)^3\times  H^1\big(S_h\times (-M,M)\big)^3$:
\begin{equation*}
\begin{cases}
\Delta^*\,\bv+\omega^2\,\bv=\mathbf{0},\hspace{5pt} \Delta^*\,\bw+\omega^2\,\bw=\mathbf{0}  &\mbox{in}\quad S_h\times (-M,M),\\[2mm]
\bw-\bv=\bff_1,\qquad\quad\mathcal{T}_{\nu}\,\bw-\mathcal{T}_{\nu}\,\bv=\bg_1 &\mbox{on}\quad\Gamma^+_{h}\times (-M,M),\\[2mm]
\bw-\bv=\bff_2,\qquad\quad\mathcal{T}_{\nu}\,\bw-\mathcal{T}_{\nu}\,\bv=\bg_2 &\mbox{on}\quad\Gamma^-_{h}\times (-M,M),
\end{cases}
\end{equation*}
can be reduced  into
\allowdisplaybreaks
\begin{equation}\label{eq:trans3D'}
\begin{cases}
\hat{\Delta}^*\,\mathcal{P}(\bw)(\bx')+\omega^2\mathcal{P}(\bw)(\bx')=G_1(\bx'), &\bx'\in S_h,\\[2mm] \hat{\Delta}^*\,\mathcal{P}(\bv)(\bx')+\omega^2\mathcal{P}(\bv)(\bx')=G_2(\bx'),  &\bx'\in S_h,\\[2mm]
\mathcal{P}(\bw)(\bx')=\mathcal{P}(\bv)(\bx')+\mathcal{P}(\bff_1)(\bx'), &\bx'\in \Gamma^+_{h},\\[2mm]
\mathcal{P}(\bw)(\bx')=\mathcal{P}(\bv)(\bx')+\mathcal{P}(\bff_2)(\bx'), &\bx'\in \Gamma^-_{h},\\[2mm]
\mathcal{R}^+_1=\mathcal{R}^+_2+\mathcal{P}(\bg)(\bx'),&\bx'\in \Gamma^+_{h},\\[2mm]
\mathcal{R}^-_1=\mathcal{R}^-_2+\mathcal{P}(\bg)(\bx'),&\bx'\in \Gamma^-_{h},
\end{cases}
\end{equation}
where
\begin{align*}
G_1=&-\int^{+L}_{-L}\phi''(x_3)\begin{bmatrix} \lambda \bv^{(1,2)}(\bx)\\
 (2\lambda+\mu)v_3(\bx)  \end{bmatrix}\mathrm{d}x_3+(\lambda+\mu)\int_{-L}^{+L}\phi'(x_3)\begin{bmatrix} \nabla v_3(\bx)\\ \partial_1 v_1(\bx)+\partial_2 v_2(\bx)  \end{bmatrix}\mathrm{d}x_3\\[2mm]
G_2=&-\int^{+L}_{-L}\phi''(x_3)\begin{bmatrix} \lambda \bw^{(1,2)}(\bx)\\ (2\lambda+\mu)w_3(\bx)  \end{bmatrix}\mathrm{d}x_3+(\lambda+\mu)\int_{-L}^{+L}\phi'(x_3)\begin{bmatrix} \nabla w_3(\bx)\\ \partial_1 w_1(\bx)+\partial_2 w_2(\bx)  \end{bmatrix}\mathrm{d}x_3\\[2mm]
\mathcal{R}^+_1=&\begin{bmatrix}
\mathcal{T}_{\nu}\mathcal{P}(\bv^{(1,2)})+\lambda\mathcal{P}(\partial_3 v_3)\nu\\[1mm]
\mu\partial_{\nu}\mathcal{P}(v_3)+\mu\begin{bmatrix}\mathcal{P}(\partial_3 v_1)\\ \mathcal{P}(\partial_3 v_2)
\end{bmatrix}\cdot \nu
\end{bmatrix},\quad\mathcal{R}^+_2=\begin{bmatrix}
\mathcal{T}_{\nu}\mathcal{P}(\bw^{(1,2)})+\lambda\mathcal{P}(\partial_3 w_3)\nu\\[1mm]
\mu\partial_{\nu}\mathcal{P}(w_3)+\mu\begin{bmatrix}\mathcal{P}(\partial_3 w_1)\\ \mathcal{P}(\partial_3w_2)
\end{bmatrix}\cdot \nu
\end{bmatrix} \,\,\mbox{on}\,\, \Gamma^+_{h},\\[2mm]
\mathcal{R}^-_1=&\begin{bmatrix}
\mathcal{T}_{\nu}\mathcal{P}(\bv^{(1,2)})+\lambda\mathcal{P}(\partial_3 v_3)\nu\\[1mm]
\mu\partial_{\nu}\mathcal{P}(v_3)+\mu\begin{bmatrix}\mathcal{P}(\partial_3 v_1)\\ \mathcal{P}(\partial_3 v_2)
\end{bmatrix}\cdot \nu
\end{bmatrix},\quad\mathcal{R}^-_2=\begin{bmatrix}
\mathcal{T}_{\nu}\mathcal{P}(\bw^{(1,2)})+\lambda\mathcal{P}(\partial_3 w_3)\nu\\[1mm]
\mu\partial_{\nu}\mathcal{P}(w_3)+\mu\begin{bmatrix}\mathcal{P}(\partial_3 w_1)\\ \mathcal{P}(\partial_3w_2)
\end{bmatrix}\cdot \nu
\end{bmatrix} \,\,\mbox{on}\,\, \Gamma^-_{h}
\end{align*}
Here, $\nu$ denotes the exterior unit normal vectors to $\Gamma^\pm_{h}$, $\mathcal{T}_{\nu}$ is the two-dimensional boundary traction operator.
\end{lem}

By applying the decomposition of $\widetilde{\Delta}^*$ given in \eqref{eq:L} along with the decompositions of $\bv$, $\bw$, $\bff_i$ and $\bg_i$ in \eqref{eq:DEC}, we can directly obtain the following results. The proof is omitted here.

\begin{lem}\nonumber
Under the same setup as in Lemma~\ref{lem:R3_1}, the PDE system \eqref{eq:trans3D'} is equivalent to the following two PDE systems:
\begin{equation}\label{eq:trans3D_1}
\begin{cases}
\hat{\Delta}^*\,\mathcal{P}(\bv^{(1,2)})(\bx')+\omega^2\mathcal{P}(\bv^{(1,2)})(\bx')=G^{(1,2)}_1(\bx')\quad &\mbox{in}\quad S_h,\\[1mm]
\hat{\Delta}^*\,\mathcal{P}(\bw^{(1,2)})(\bx')+\omega^2\mathcal{P}( \bw^{(1,2)})(\bx')=G^{(1,2)}_2(\bx') \quad &\mbox{in}\quad S_{h},\\[1mm]
\mathcal{P}(\bw^{(1,2)})(\bx')=\mathcal{P}(\bv^{(1,2)})(\bx')+\mathcal{P}(\bff^{(1,2)}_1)(\bx') \quad&\mbox{on}\quad \Gamma^+_{h},\\[1mm]
\mathcal{P}(\bw^{(1,2)})(\bx')=\mathcal{P}(\bv^{(1,2)})(\bx')+\mathcal{P}(\bff^{(1,2)}_2)(\bx') \quad&\mbox{on}\quad \Gamma^-_{h},\\[1mm]
\mathcal{T}_{\nu}\mathcal{P}(\bw^{(1,2)})
=\mathcal{T}_{\nu}\mathcal{P}(\bv^{(1,2)})+\mathcal{P}(\bg^{(1,2)}_1)(\bx')\quad&\mbox{on}\quad \Gamma^+_{h},\\[1mm]
\mathcal{T}_{\nu}\mathcal{P}(\bw^{(1,2)})
=\mathcal{T}_{\nu}\mathcal{P}(\bv^{(1,2)})+\mathcal{P}(\bg^{(1,2)}_2)(\bx')\quad&\mbox{on}\quad \Gamma^-_{h}
\end{cases}
\end{equation}
and
\begin{equation}\label{eq:trans3D_2}
\begin{cases}
\lambda\Delta'\,\mathcal{P}(v_3)(\bx')+\omega^2\mathcal{P}( v_3)(\bx')=G^{(3)}_1(\bx')&\mbox{in}\quad  S_h,\\[1mm] \lambda\Delta'\,\mathcal{P}(w_3)(\bx')+\omega^2\mathcal{P}( w_3)(\bx')=G^{(3)}_2(\bx')  &\mbox{in}\quad  S_h,\\[1mm]
\mathcal{P}(w_3)(\bx')=\mathcal{P}(v_3)(\bx')+\mathcal{P}(f^{3}_1)(\bx')&\mbox{on}\quad  \Gamma^+_{h},\\[1mm]
\mathcal{P}(w_3)(\bx')=\mathcal{P}(v_3)(\bx')+\mathcal{P}(f^{3}_2)(\bx')&\mbox{on}\quad  \Gamma^-_{h},\\[1mm]
\partial_{\nu}\mathcal{P}(w_3)(\bx')
=\partial_{\nu}\mathcal{P}(v_3)(\bx')+\frac{1}{\mu}\mathcal{P}(g^{3}_1)(\bx')\qquad\quad&\mbox{on}\quad  \Gamma^+_{h},\\[1mm]
\partial_{\nu}\mathcal{P}(w_3)(\bx')
=\partial_{\nu}\mathcal{P}(v_3)(\bx')+\frac{1}{\mu}\mathcal{P}(g^{3}_2)(\bx')\qquad\quad&\mbox{on}\quad  \Gamma^-_{h},
\end{cases}
\end{equation}where
\allowdisplaybreaks
\begin{align*}
&G^{(1,2)}_1=-\lambda \int^{+L}_{-L}\phi''(x_3) \bv^{(1,2)}(\bx)\mathrm{d}x_3+(\lambda+\mu)\int_{-L}^{+L}\phi'(x_3)\nabla v_3(\bx)\mathrm{d}x_3,\\[1mm]
&G^{(1,2)}_2=- \lambda \int^{+L}_{-L}\phi''(x_3) \bw^{(1,2)}(\bx)\mathrm{d}x_3+(\lambda+\mu)\int_{-L}^{+L}\phi'(x_3)\nabla w_3(\bx)\mathrm{d}x_3,\\[1mm]
&G^{(3)}_1=- (2\lambda+\mu)\int^{+L}_{-L}\phi''(x_3) v_3(\bx)\mathrm{d}x_3+(\lambda+\mu)\int_{-L}^{+L}\phi'(x_3)(\partial_1 v_1+\partial_2 v_2) \mathrm{d}x_3,\\[1mm]
&G^{(3)}_2=- (2\lambda+\mu)\int^{+L}_{-L}\phi''(x_3) w_3(\bx)\mathrm{d}x_3+(\lambda+\mu)\int_{-L}^{+L}\phi'(x_3)(\partial_1 w_1+\partial_2 w_2) \mathrm{d}x_3.
\end{align*}
\end{lem}

As discussed in Remark~4.2 of \cite{DLS2021}, the regularity result for the underlying elastic displacement around a general polyhedral corner in $\mathbb{R}^3$ is challenging. Therefore, we focus exclusively on 3D edge corners. To obtain the continuity of $\mathcal{P}(f^{3})$ and $\mathcal{P}(g^{3})$ at a 3D edge corner, we also utilize another CGO solution $u$. Specifically, the CGO solution $u_0$ is introduced in \cite{Blasten2018} with similar forms and properties as those in \cite{DCL2021}.

\begin{lem}\label{lem:cgo2} \cite[Lemma 2.2]{Blasten2018}
Let $\bx'=(x_1,x_2)^t=r(\cos\theta,\sin\theta)^t\in \mathbb{R}^2$ and $s\in \mathbb{R}_+$,
\beq\label{cgo21}
u_0(s\bx'):=\exp (-\sqrt{sr}Z(\theta)),
\eeq
where $Z(\cdot):=\cos(\theta/2) +\mathrm {i }\sin( \theta/2 )=e^{\mathrm{i}\theta/2}$. Then $s\longmapsto u_0(s\bx')$ decays exponentially in $\mathbb{R}_+$ and
\beq\nonumber
\Delta' u_0=0\quad\mbox{in}\quad \mathbb{R}^2\backslash\mathbb{R}^2_{0,-}
\eeq
where $\mathbb{R}^2_{0,-}:=\{\bx'\in\mathbb{R}^2|\bx'=(x_1,x_2)^t;x_1\leq 0,x_2=0\}$. Moreover,
\beq
\nonumber
\int_\mathcal {A} u_0(s\bx'){\rm d}\bx'=6{\rm i}(e^{-2\theta_{\max}\rm i}-e^{-2\theta_{\min} \mathrm i})s^{-2}
\eeq
and for $\alpha,\,s>0$ and $h>0$
\allowdisplaybreaks
\begin{align*}
\int_{\mathcal {A}}|u_0(s\bx')||\bx'|^{\alpha}{\rm d}\bx'&\leq\frac{2(\theta_{\max}-\theta_{\min})\Gamma(2\alpha+4)}{\delta_{\mathcal{A}}^{2\alpha+4}}s^{-\alpha-2},\nonumber\\
 \int_{\mathcal {A}\backslash B_h}|u_0(s\bx')|{\rm d}\bx'&\leq\frac{6(\theta_{\max}-\theta_{\min})}{\delta_{\mathcal{A}}^4}s^{-2}e^{-\frac{\sqrt{sh}}{2}\delta_{\mathcal{A}}}, 
\end{align*}
where $\mathcal {A}$ is defined  in Section~\ref{se:pre} and $\delta_{\mathcal{A}}:=\min\limits_{\theta_{\min}<\theta<\theta_{\max}} \cos \frac{\theta}{2}$ is a positive constant.
\end{lem}

\begin{lem}\nonumber\cite[Lemma 2.4]{DCL2021}
For any $\alpha>0$, if $\omega(\theta)>0$, then we have
$$
\lim\limits_{s\rightarrow +\infty}\int^h_0 r^{\alpha}e^{-\sqrt{sr}\omega(\theta)}\mathrm{d}\mathrm{r}=\Oh(s^{-\alpha-1}).
$$
\end{lem}

\begin{lem}\label{10-lem:23'}\cite[Lemma~2.3]{DCL2021}
	Let $s_h$ be defined as in \eqref{eq:ball1} and $u_0$ be given in \eqref{cgo21}. Then $u_0 \in H^1(S_h)^2$ and $\Delta' u_0 = 0$ in $S_h$. Furthermore, it holds that
	\begin{equation*}
	\big\|u_0\big\|_{L^2(S_h)^2 } \leq 	 \frac{\sqrt{\theta_{\max}-\theta_{\min}} e^{- 2\sqrt{s\theta} \,\delta_{\mathcal{A}} }h^2}{2}
	\end{equation*}
	and
	\begin{equation*}
		\Big  \||\bx'|^\alpha  u_0 \Big \|^2_{L^{2}(s_h )^2  }\leq s^{-2(\alpha+1 )} \frac{2(\theta_{\max}-\theta_{\min})\Gamma(4\alpha+4)    }{(4\delta_\mathcal{A})^{2\alpha+2  } } \,
	\end{equation*}
where $ \zeta  \in [0,h ]$ and $\delta_\mathcal{A}$ is given in Lemma~\ref{lem:cgo2}.
\end{lem}

We derive a key proposition to establish a primary geometric result, which is a three-dimensional result analogous to Proposition~\ref{prop:trans1}.
\begin{prop}\label{prop:trans2}
Consider the same setup as described in Lemma~\ref{lem:R3_1}, with the point $\bx_c$ coinciding with the origin. Let $\bv\in H^1\left(S_h\times (-M,M)\right)^3$ and $\bw\in H^1\left(S_h\times (-M,M)\right)^3$ satisfy the PDE system described in \eqref{eq:trans3D'}, where $\bff_1\in C^{1,\alpha_1}\big(\Gamma^+_h\times (-M,M)\big)^3\cap H^{1/2}\big(\Gamma^+_h\times (-M,M)\big)^3$, $\bff_2\in C^{1,\alpha_2}\big(\Gamma^-_h\times (-M,M)\big)^3\cap H^{1/2}\big(\Gamma^-_h\times (-M,M)\big)^3$, $\bg_1\in C^{\beta_1}\big(\Gamma^+_h\times (-M,M)\big)^3\cap L^2\big(\Gamma^+_h\times (-M,M)\big)^3$, and $\bg_2\in C^{\beta_2}\big(\Gamma^-_h\times (-M,M)\big)^3\cap L^2\big(\Gamma^-_h\times (-M,M)\big)^3$, all of which are independent of $x_3$. Here, $\alpha_i$ and $\beta_i$ are in $(0,1)$ for $i=1,2$. Then we have
 \beq\label{eq:vansh2}
 \bff_1({\bf 0})=\bff_2({\bf 0}).
 \eeq
Additionally, if the condition $\nabla
 \bff_1({\bf 0})=\nabla \bff_2({\bf 0})={
 \bf 0}$ holds, then we obtain
\beq \label{eq:vansh2-2}
\bg^{(1,2)}_1({\bf 0})=\Theta\bg^{(1,2)}_2({\bf 0})\quad{and}\quad g_1^{3}({\bf 0})=g_2^{3}({\bf 0})=0,
 \eeq
where $$\Theta=\begin{bmatrix}-\cos(\theta_{\max}-\theta_{\min}),&-\sin(\theta_{\max}-\theta_{\min})\\
    -\sin(\theta_{\max}-\theta_{\min}), &+\cos(\theta_{\max}-\theta_{\min})
    \end{bmatrix}$$
    with $\theta_{\max}$ and $\theta_{\min}$ denoting the arguments corresponding to the boundaries $\Gamma^+_h$ and $\Gamma^-_h$, respectively.
\end{prop}
\begin{proof}
As in the proof of Proposition~\ref{prop:trans1}, we focus solely on the corresponding proofs for $\Re\mathcal{P}(\bv)$ and $\Re\mathcal{P}(\bw)$. We employ similar arguments as in the proof of Proposition~\ref{prop:trans1}, with necessary modifications. We will divide the proof into two steps.

\medskip  \noindent {\it  Step I.}~First we shall prove that
\begin{equation}\label{eq:vansh_2}
 \Re\mathcal{P}(\bff^{(1,2)}_1)({\bf 0})=\Re\mathcal{P}(\bff^{(1,2)}_2)({\bf 0}) \quad \mbox{and}\quad \bg^{(1,2)}_1({\bf 0})=W\bg^{(1,2)}_2({\bf 0}).
 \end{equation}
In this part, we consider the PDE system~\eqref{eq:trans3D_1}. The Betti's second formula directly yields the following identity
 \begin{align*}
    &\int_{S_h} \left(\Re G_1^{(1,2)}-\Re G_2^{(1,2)}\right)\cdot \bu_0\,\mathrm{d}\bx'-\int_{\Lambda_h}
    \mathcal{T}_{\nu}\Re\mathcal{P}\big(  \bv^{(1,2)}-\bw^{(1,2)}\big)\cdot \bu_0-\Re\mathcal{P}\big(\bv-\bw\big)\cdot \mathcal{T}_{\nu}\bu_0\mathrm{d}\sigma\\
    &=\int_{\Gamma_h^+}\Re\mathcal{P}(\bff_1^{(1,2)})\cdot\mathcal{T}_{\nu}\bu_0  -\Re\mathcal{P}(\bg_1^{(1,2)})+\int_{\Gamma_h^-}\Re\mathcal{P}(\bff_2^{(1,2)})\cdot\mathcal{T}_{\nu}\bu_0  -\Re\mathcal{P}(\bg_2^{(1,2)})\cdot \bu_0\mathrm{d}\sigma.
\end{align*}
Combining the regularities with respect to $\bff_i$ and $\bg_i$ on $\Gamma^\pm_h\times (-M,M)$ and the fact that $\bff_i$ and $\bg_i$ do not depend on $x_3$, $i=1,2$, we have the following expansions
\begin{align*}
 \mathcal{P}(\bff_1)(\bx')&=\mathcal{P}(\bff_1)({\bf0})+\nabla \mathcal{P}(\bff_1)({\bf0})\cdot \bx'+ \delta\mathcal{P}(\bff_1)(\bx'),\quad \big| \delta\mathcal{P}(\bff_1)(\bx') \big|\leq \|\mathcal{P}(\bff_1)\|_{C^{\alpha_1}}|\bx'|^{1+\alpha_1},  \\
  \mathcal{P}(\bff_2)(\bx')&=\mathcal{P}(\bff_2)({\bf0})+ \nabla \mathcal{P}(\bff_1)({\bf0})\cdot \bx'+\delta\mathcal{P}(\bff_2)(\bx'),\quad \big| \delta\mathcal{P}(\bff_2)(\bx') \big|\leq \|\mathcal{P}(\bff_2)\|_{C^{\alpha_2}}|\bx'|^{1+\alpha_2},  \\
 \mathcal{P}(\bg_1)(\bx')&=\mathcal{P}(\bg_1)({\bf0})+ \delta\mathcal{P}(\bg_1)(\bx'),\,\qquad\qquad\qquad\qquad \big| \delta\mathcal{P}(\bg_1)(\bx')\big|\leq \|\mathcal{P}(\bg_1)\|_{C^{\beta_1}}|\bx'|^{\beta_1},\\
  \mathcal{P}(\bg_2)(\bx')&=\mathcal{P}(\bg_2)({\bf0})+ \delta\mathcal{P}(\bg_2)(\bx'),\,\qquad\qquad \qquad\qquad\big| \delta\mathcal{P}(\bg_2)(\bx')\big|\leq \|\mathcal{P}(\bg_2)\|_{C^{\beta_2}}|\bx'|^{\beta_2}.
 \end{align*}
 By a series of derivations similar to Proposition~\ref{prop:trans1}, we can obtain the following integral identity
{\small \begin{align*}
&\Big(\frac{\Re\mathcal{P}(\bg^{(1,2)}_1)({\bf 0})\cdot\eeta }{-\bxi\cdot \hat{\bx}_{\max}}+\frac{\Re\mathcal{P}(\bg^{(1,2)}_2)({\bf 0})\cdot\eeta }{-\bxi\cdot \hat{\bx}_{\min}}\Big)+2\mu\Big(\Re\mathcal{P}(\bff^{(1,2)}_1)({\bf 0})\cdot\,\frac{\eeta\,\bxi\cdot\nu_{\max}}{\bxi\cdot\nu_{\max}} +\Re\mathcal{P}(\bff^{(1,2)}_2)({\bf 0})\cdot\frac{\eeta\,\bxi\cdot\nu_{\min}}{\bxi\cdot\nu_{\min}}        \Big)\nonumber\\
&=\sum\limits_{j=1}^{15}R_j,
\end{align*}}where $\Re\bff_i({\bf 0})$, $\Re\bg_i({\bf 0})$, $\Re\delta\bff_i$, $\Re\delta\bg_i$, $\Re\bv$ and $\Re\bw$ in $R_1$--$R_{14}$ given by \eqref{eq:iden1} are replaced by $\Re\mathcal{P}(\bff^{(1,2)}_i)({\bf 0})$, $\Re\mathcal{P}(\bg^{(1,2)}_i)({\bf 0})$, $\delta\Re\mathcal{P}(\bff^{(1, 2)}_i)$, $\delta\Re\mathcal{P}(\bg^{(1,2)}_i)$, $\Re\mathcal{P}(\bv^{(1,2)}_i)$ and $\Re\mathcal{P}(\bw^{(1,2)}_i)$ with $i=1,2$.
In addition,
\begin{align*}
R_{15}=\int _{S_h}\big(\Re G^{(1,2)}_2-\Re G^{(1,2)}_1\big)\cdot \bu_0\,\mathrm{d}{\bx'}. 
\end{align*}
Similar to Proposition~\ref{prop:trans1}, it yields that
\begin{alignat*}{3}
&\big|R_1\big|=\Oh(\tau^{-1}e^{-c'_1\tau}),\hspace{12pt}&\big|R_2\big|&=\Oh(e^{-c'_2\,\tau}), \hspace{12pt}&\big|R_3\big|&=\Oh(\tau^{-1}e^{-c'_3\tau}),\\
&\big|R_4\big|=\Oh(e^{-c'_4\,\tau}),\hspace{12pt}&\big|R_5\big|&=\Oh(\tau^{-1}e^{-c'_5\,\tau}),\hspace{12pt}
&\big|R_7\big|&=\Oh(\tau^{-1}e^{-c'_7\,\tau}),\\
&\big|R_6\big|=\Oh(\tau^{-1-\alpha_1} ),\hspace{12pt}&\big|R_8\big|&= \Oh(\tau^{-1-\alpha_2} ),\hspace{14pt}& \big|R_{10}\big|&= \Oh(\tau^{-\beta_1-1} ),\\
&\big|R_{12}\big|= \Oh(\tau^{-\beta_2-1} ),\hspace{12pt}&\big|R_{11}\big|&= \Oh(\tau^{-2} ),\hspace{12pt}&\big|R_{13}\big|&= \Oh(\tau^{-2}),\\
&\big|R_{9}\big|= \Oh(e^{-\tau\zeta_0h\tau}),\hspace{12pt}&\big|R_{14}\big|&= \Oh(\tau^{-2} ),\hspace{12pt}&&
\end{alignat*}
where these constants above do not depend on $\tau$ and $\zeta_0,\alpha_1,\alpha_2$, $\beta_1$ and $\beta_2$ belong to $(0,1)$.
For $R_{15}$,  write $G^{(1,2)}_1-G^{(1,2)}_2:=\bh_1+\bh_2$, where $\bh_1:=-\int_{-L}^{+L}\phi''(x_3)\begin{bmatrix}\lambda(v_1-w_1)\\ \lambda(v_2-w_2)\end{bmatrix}\mathrm{d}x_3$ and $\bh_2:=(\lambda+\mu)\int_{-L}^{+L}\phi'(x_3)\begin{bmatrix}\partial_1(v_3-w_3)\\ \partial_2(v_3-w_3)\end{bmatrix}\mathrm{d}x_3$. By the regularities of $\bv$ and $\bw$ in $S_h\times (-M,M)$, we directly obtain that $G^{(1,2)}_1\in H^1(S_h)^2$ and $G^{(1,2)}_2\in L^2(S_h)^2$. By the Cauchy-Schwarz inequality and the same method to prove \eqref{cgo2-1}, we prove
\begin{align*}
&\bigg|\int_{S_h}\big(G^{(1,2)}_1-G^{(1,2)}_2\big)\cdot \bu_0\,\mathrm{d}\bx'\,\bigg|=\bigg|\int_{S_h}(\bh_1+\bh_2)\cdot \bu_0\,\mathrm{d}\bx'\,\bigg|\\[1mm]
&\,\,\leq \|\bh_1\|_{L^2(S_h)^2}\|\bu_0\|_{L^2(S_h)^2}+\|\bh_2\|_{L^2(S_h)^2}\|\bu_0\|_{L^2(S_h)^2}\\[1mm]
&\,\,\leq  C'|\Re\xi|^{-2},
\end{align*}
where $C'>0$ does not depend on $\tau$. It directly implies that $\big|R_{15}\big|=\Oh(\tau^{-2} )$.
Therefore,  we can establish the following identity
\beq\label{3D_iden1}
\lim\limits_{\tau\rightarrow +\infty}\,2\mu\bigg(\Re\mathcal{P}(\bff^{(1,2)}_1)({\bf 0})\cdot\,\frac{\eeta\,\bxi\cdot\nu_{\max}}{\bxi\cdot\nu_{\max}} +\Re\mathcal{P}(\bff^{(1,2)}_2)({\bf 0})\cdot\frac{\eeta\,\bxi\cdot\nu_{\min}}{\bxi\cdot\nu_{\min}}        \bigg)=0.
\eeq

Moreover, combining with the condition that $\nabla\Re\mathcal{P}(\bff_1)({\bf 0})=\nabla\Re\mathcal{P}(\bff_2)({\bf 0})={\bf 0}$, and letting $\tau$ going to $\infty$,
we have
\beq\label{3D_iden2}
\lim\limits_{\tau\rightarrow +\infty}\,\eeta\cdot\bigg(\frac{\tau \Re\mathcal{P}(\bg^{(1,2)}_1)({\bf 0}) }{-\bxi\cdot \hat{\bx}_{\max}}+\frac{\tau \Re\mathcal{P}(\bg^{(1,2)}_2)({\bf 0}) }{-\bxi\cdot \hat{\bx}_{\min}}\bigg)=0
\eeq

Dealing with  \eqref{3D_iden1} and \eqref{3D_iden2} like \eqref{eq:R4} and \eqref{eq:R5}, it yields that \eqref{eq:vansh_2} holds.

\medskip  \noindent {\it  Step II.}~We next shall prove that
\begin{equation}\label{eq:vansh_3}
 \Re\mathcal{P}(f_1^{3})({\bf 0})= \Re\mathcal{P}(f_2^{3})({\bf 0})\quad \mbox{and}\quad \Re\mathcal{P}(g_1^{3})({\bf 0})=\Re\mathcal{P}(g_2^{3})({\bf 0})=0.
 \end{equation}
Conversely, we will use the CGO solution $u_0$ from Lemma~\ref{lem:cgo2} to establish the vanishing of $\Re\mathcal{P}(f_1^{3})$ at the same 3D edge corner. Additionally, from similar calculations, we obtain the following integral identity, where the second equation requires additional assumptions(i.e., $\nabla\mathcal{P}(f^{3}_1)({\bf 0})=\nabla\mathcal{P}(f^{3}_2)({\bf 0})={\bf 0}$). In this section, we consider the PDE system~\eqref{eq:trans3D_2}. We will deduce similar operations as above using the CGO solution
 $u_0$ provided in \eqref{lem:cgo2}, setting up an integral identity as follows:
 \begin{align*}
 \frac{1}{\lambda}\int_{S_h}\big(\Re G^{(3)}_1(\bx')-\Re G^{(3)}_2&(\bx')\big)\,u_0\,\mathrm{d}\bx'=\int_{\Lambda_h}\partial_{\nu}\mathcal{P}(v_3-w_3)\,u_0-\partial_{\nu}u_0\,\mathcal{P}(v_3-w_3)\mathrm{d}\sigma\\
&+\int_{\Gamma^+_h}{\nu_{\max}}\cdot \Big(\Re\mathcal{P}(\bff^{(1,2)}_1)(\bx') +\frac{1}{\mu}\Re\mathcal{P}(g^{3}_1)\Big)u_0-\Re\mathcal{P}(f^{3}_1)\,\,\partial_{\nu_{\max}}u_0\mathrm{d\sigma}\\
&+\int_{\Gamma^{-}_h}{\nu_{\min}}\cdot \Big(\Re\mathcal{P}(\bff^{(1,2)}_2)(\bx') +\frac{1}{\mu}\Re\mathcal{P}(g^{3}_2)\Big)\,u_0-\Re\mathcal{P}(f^{3}_2)\,\,\partial_{\nu_{\min}}u_0\mathrm{d}\sigma.
 \end{align*}
Due to  the expansions of $\bff_i$ and $\bg_i$, $i=1,2$, in \eqref{eq:DEC}, the above integral identity can be reduced into
\allowdisplaybreaks
\beq\label{ID3}
\lambda\,\mu\,\rm{i}\,\Big(\Re\mathcal{P}(f^{3}_1)({\bf 0}) -\Re\mathcal{P}(f^{3}_2)({\bf 0})-2\,\lambda\,s^{-1}\bigg(\frac{\Re\mathcal{P}(g^{3}_1)({\bf 0})}{Z^2(\theta_{\max})}+\frac{\Re\mathcal{P}(g^{3}_2)({\bf 0})}{Z^2(\theta_{\min})} \bigg)= \sum^{11}_{j=1}\widetilde{R}_j,
\eeq
 where
\begin{align*}
&\widetilde{R}_1=2\lambda \,s^{-1}\Re\mathcal{P}(g^{3}_1)({\bf 0})\Big(Z^{-1}(\theta_{\max})\,\sqrt{sh}\,e^{-\sqrt{sh}\,Z(\theta_{\max})}+Z^{-2}(\theta_M)\,e^{-\sqrt{sh}\,Z(\theta_{\max})}\Big),\\
&\widetilde{R}_2=2\lambda \,s^{-1}\Re\mathcal{P}(g^{3}_2)({\bf 0})\Big(Z^{-1}(\theta_{\min})\,\sqrt{sh}\,e^{-\sqrt{sh}\,Z(\theta_{\min})}+Z^{-2}(\theta_{\min})\,e^{-\sqrt{sh}\,Z(\theta_{\min})}\Big),\\
&\widetilde{R}_3=-\lambda\,\mu\,\int_{\Lambda_h}\partial_{\nu}u_0\,\,\Re\mathcal{P}(v_3-w_3)(\bx')  -\partial_{\nu}u_0\, \,\Re\mathcal{P}(v_3-w_3)(\bx')\,\mathrm{d}\sigma,\\
&\widetilde{R}_4=\lambda\,\mu\,\mathrm{i}\,\bigg(\frac{\Re\mathcal{P}(f^{3}_2)({\bf 0})}{e^{\sqrt{sh}\,Z(\theta_{\min})}}-\frac{\Re\mathcal{P}(f^{3}_1)({\bf 0})}{e^{\sqrt{sh}\,Z(\theta_{\max})}}\bigg),\qquad \qquad \widetilde{R}_5=\mu\int_{S_h}(t_1+t_2)\,u_0\mathrm{d}\bx',\\
&\widetilde{R}_6=\lambda\,\mu\,\int_{\Gamma^+_h}\delta\Re\mathcal{P}(f^{3}_1)(\bx')\,\, \partial_{\nu_{\max}}u_0\,\mathrm{d}\sigma,\quad \quad\, \widetilde{R}_7=\lambda\,\mu\,\int_{\Gamma^-_h}\delta\Re\mathcal{P}(f^{3}_2)(\bx')\,\, \partial_{\nu_{\min}}u_0\,\mathrm{d}\sigma,\\
&\widetilde{R}_{10}=\lambda\,\mu\,\int_{\Gamma^+_h}\left(\nabla\Re\mathcal{P}(f^{3}_1)(\bx')\cdot \bx'\right)\,\, \partial_{\nu_{\max}}u_0\,\mathrm{d}\sigma,\quad
\widetilde{R}_8=-\lambda\int_{\Gamma^+_h}\delta\Re\mathcal{P}(g^{3}_1)(\bx')\,\, u_0\,\mathrm{d}\sigma,\\
&\widetilde{R}_{11}=\lambda\,\mu\,\int_{\Gamma^-_h}\left(\nabla\Re\mathcal{P}(f^{3}_2)(\bx')\cdot \bx'\right)\,\, \partial_{\nu_{\min}}u_0\,\mathrm{d}\sigma,\quad \widetilde{R}_9=-\lambda\int_{\Gamma^-_h}\delta\Re\mathcal{P}(g^{3}_2)(\bx')\,\, u_0\,\mathrm{d}\sigma.
\end{align*}
Here,
\allowdisplaybreaks
\begin{align*}
t_1&=-(2\lambda+\mu)\,\Re\int_{-L}^L\phi''(\bx')(v_3-w_3)\mathrm{d}x_3,\\
t_2&=(\lambda+\mu)\,\Re\int_{-L}^L\phi'(\bx')\big(\partial_1(v_1-w_1)+\partial_2(v_2-w_2)\big)\mathrm{d}x_3.
\end{align*}
Using those estimates list in Lemma~\ref{lem:cgo2}--Lemma~\ref{10-lem:23'}, we have
\begin{alignat*}{3}
&\big|\widetilde{R}_1\big|=\Oh(e^{-q_1 \sqrt{s}}),&\quad&\big|\widetilde{R}_2\big|=\Oh(e^{-q_2 \sqrt{s}}),&\quad&\big|\widetilde{R}_3\big|=\Oh(s^{-1} e^{-1/2 h s}),\,\big|\widetilde{R}_4\big|=\Oh(s^{-q_4\sqrt{s}}),\\
&\big|\widetilde{R}_6\big|=\Oh(s^{-\alpha_+-1}),&\quad&\big|\widetilde{R}_7\big|=\Oh(e^{-\alpha_--1}),&\quad&\big|\widetilde{R}_8\big|=\Oh(e^{-\beta_+-1}),\,\,\quad\big|\widetilde{R}_9\big|=\Oh(e^{-\beta_--1}),\\
&\big|\widetilde{R}_{10}\big|=\Oh(s^{-1}),&\quad&\big|\widetilde{R}_{11}\big|=\Oh(s^{-1}).&\quad&
\end{alignat*}
where these above constants do not depend on $s$. Using a similar technique of estimating $Q_8$, we get
$$
\big|\widetilde{R}_5\big|=\Oh(e^{-q_5 s}),\quad \mbox{where} \quad q_5>0.
$$
Let $s \rightarrow +\infty$, it is direct to obtain the first equation in \eqref{eq:vansh_3}.

It is noted that $\nabla\mathcal{P}(f^{3}_j)({\bf 0})={\bf 0}$. We substitute the above equation into \eqref{ID3} and then multiply the resulting identity by $s$, yielding
$$
\frac{\Re\mathcal{P}(g^{3}_1)({\bf 0})}{Z^2(\theta_{\max})}+\frac{\Re\mathcal{P}(g^{3}_2)({\bf 0})}{Z^2(\theta_{\min})}=0.
$$
 It's worth noting that $\frac{Z^2(\theta_{\max})}{Z^2(\theta_{\min})}=e^{\mathrm{i}(\theta_{\max}-\theta_{\min})}$ and $\theta_{\max}-\theta_{\min}\in (0,\pi)$. Therefore, the second equation in \eqref{eq:vansh_3} is valid.

Thanks to the symmetric roles of  $(\Re\bv,\Re\bw)$ and $(\Im\bv,\Im\bw)$, we can easily derive the two equations \eqref{eq:vansh2} and \eqref{eq:vansh2-2}.

\end{proof}

 \section{Proofs of Theorem~\ref{th:main_loca}--Theorem~\ref{th:main2}}\label{proofs}

This section provides detailed proofs for Theorems~\ref{th:main_loca}--\ref{th:main2}. Before proceeding, we will review some key properties of the CGO solutions introduced in \cite{BlastenLin2019}.
\begin{lem}
Let $\Omega\Subset \mathcal{C}$ and $\Omega\cap(\mathbb{R}_{-}\cup\{ 0
\})=\emptyset$. For the $\bx=(x_1,x_2)^t\in \Omega$, denote $z=x_{1}+{\rm i}\,x_{2}$ and $\theta=\mathrm{arg}(z)$. Let
    \begin{equation}\label{eq:lame3}
\widetilde{\bu}_0( \bx ):=
\begin{pmatrix}
    \exp(-s\sqrt{z})\\
    {\rm i}\exp(-s\sqrt{z})
    \end{pmatrix}:=\begin{pmatrix}
    \widetilde{u}^0_1( \bx)\\
    \widetilde{u}^0_2( \bx)
\end{pmatrix} ,
\end{equation}
 where  $s\in \mathbb{R}_{+}$. The complex square root of $z$ is defined as
\begin{equation}\notag 
\sqrt{z}=\sqrt{|z|}\left(\cos\frac{\theta}{2}+{\rm i}\sin\frac{\theta}{2}\right ),
\end{equation}
where $-\pi<\theta<\pi$ is the argument of $z$. Then $\widetilde{\bu}_0$ satisfies $\Delta^*\,\bu={\bf 0}$ in $\Omega$.
Moreover, in the open sector $\mathcal{A}$ defined in \eqref{eq:ball1},  we have the following properties
\beq\nonumber
\int_\mathcal{A} \widetilde{u}^0_1(\bx){\rm d}\bx=6{\rm i}(e^{-2\theta_{\max}{\rm i}}-e^{-2\theta_{\min} \mathrm i})s^{-4}
\eeq
and
\beq\label{eq:lame6}
\int_{\mathcal{A}}|\widetilde{u}^0_{j}(\mathbf{x})||\mathbf{x}|^{\alpha}{\rm d}\mathbf{x}\leq\frac{2(\theta_{\max}-\theta_{\min})\Gamma(2\alpha+4)}{\delta_{\mathcal{A}}^{2\alpha+4}}s^{-2\alpha-4},\quad j=1,2,
\eeq
where $\mathcal{A}$ is defined in \eqref{corner}, $\alpha,h>0$, $\delta_{\mathcal{A}}=\min\limits_{\theta_{\min}<\theta<\theta_{
\max
}} \cos \frac{\theta}{2}$ is a positive constant.
\end{lem}

The following lemmas state significant properties and regularity results of the CGO solution $\widetilde{\bu}_0$,  which are beneficial for the subsequent analysis.




\begin{lem}\label{10-lem:23}\cite[Lemma~2.3]{DLS2021}
	Let $S_{\bx_c,_h}$ be defined in \eqref{eq:ball1} and $\widetilde{\bu}_0$ be given in \eqref{eq:lame3}. Then $\widetilde{\bu}_0 \in H^1(S_{\bx_c,h})^2$ and $\Delta^*\, \widetilde{\bu}_0 =\mathbf 0$ in $S_{\bx_c,h}$. Furthermore, it holds that
	\begin{align*}
\big\|\widetilde{\bu}_0\big\|_{L^2(S_{\bx_c,_h})^2 } &\leq 	 \sqrt{\theta_{\max}-\theta_{\min}} e^{- s\sqrt{\Theta }\, h },\\[1mm]
\Big  \||\mathbf x|^\alpha \widetilde{\bu}_0 \Big \|_{L^{2}(S_{\bx_c,_h} )^2  }&\leq s^{-2(\alpha+1 )} \frac{2\sqrt{(\theta_{\max}-\theta_{\min})\Gamma(4\alpha+4) }   }{(2\delta_{\mathcal{A}})^{2\alpha+2  } } ,
	\end{align*}
where $ \Theta  \in [0,h ]$ and $\delta_{\mathcal{A}}$ is defined in \eqref{eq:lame6}.
\end{lem}
 \begin{lem}\cite[Lemma 2.8]{DCL2021}\label{10-lem:u0 int}
	Let $\Gamma^\pm_h$ and $\widetilde{u}^0_1(\mathbf x)$ be respectively defined in \eqref{eq:ball1} and \eqref{eq:lame3} with $\bx_c$ coinciding with the origin. We have	
 \begin{align*}
\int_{\Gamma^+_{h} }  \widetilde{u}^0_1 (\mathbf x)  {\rm d} \sigma &=2 s^{-2}\left( \hat{\mu}(\theta_{\max} )^{-2}-   \hat{\mu}(\theta_{\max} )^{-2} e^{ -s\sqrt{h} \,\hat{\mu}(\theta_{\max} ) }-  \hat{\mu}(\theta_{\max} )^{-1} s\sqrt{h}\,   e^{ -s\sqrt{h}\, \hat{\mu}(\theta_{\max} ) }  \right  ),  \\[2pt]
 \int_{\Gamma^-_{h}  }  \widetilde{u}^0_1 (\mathbf x)  {\rm d} \sigma &=2 s^{-2} \Big( \hat{\mu}(\theta_{\min})^{-2}-   \hat{\mu}(\theta_{\min} )^{-2} e^{ -s\sqrt{h}\mu(\theta_{\min} )} -  \hat{\mu}(\theta_{\min} )^{-1} s\sqrt{h} \,  e^{ -s\sqrt{h}\,\hat{\mu}(\theta_{\min} ) }  \Big  ),	
 \end{align*}
 	where $ \hat{\mu}(\theta ):=\cos(\theta/2) +\mathrm i \sin( \theta/2 )=e^{\mathrm{i}\theta/2}$.
 \end{lem}

We are in the position to give the proofs of Theorem~\ref{th:main_loca}--Theorem~\ref{th:main2}.

\begin{proof}[Proof of Theorem~\ref{th:main_loca}]
To prove this theorem by contradiction, we will consider two cases.

\medskip \noindent In {\bf Case I}, we assume the existence of a planar corner or a 3D edge corner, denoted by $\bx_c$, located on $\Sigma_1\Delta\Sigma_2$ within the domain $\Omega^1_1\cap\Omega^2_1$. Without loss of generality, we assume that $\bx_c$ coincides with the origin $\bf 0$, where $\bf 0\in \Sigma_1$ and $\bf 0\notin \Sigma_2$. We define $\bw=\bu_1-\bu_2$.

\medskip  \noindent {\it Step I:} Establish the local uniqueness of the fault $\Sigma_1$ and $\Sigma_2$, specifically that the set $\Sigma_1 \Delta \Sigma_2$ cannot contain a planar corner or a 3D edge corner.

\medskip \noindent In the 2D case, we consider  $\Gamma^\pm_{h}=\Sigma_1\cap B_h$, $S_{h}=\Omega_-\cap B_h$, and $S_{h}\cap\Omega_+\neq \emptyset$ for sufficiently small $h\in\mathbb{R}_+$. We define $\Omega_+= \Omega\backslash\overline{\Omega}_-$, where $\Omega_-$ is defined similarly to (i) in Definition \ref{def:Admis1-2}. Since $\bu_1\big|_{\partial\Omega_0}=\bu_2\big|_{\partial \Omega_0}$ and $\partial\Omega_0\Subset\partial\Omega_2$, we have $\bw=\bu_1-\bu_2=\mathcal{T}_{\nu}\bu_1-\mathcal{T}_{\nu}\bu_2={\bf 0}$ on $\partial\Omega_0$. Let $\bw^{-}$ and $\bw^{+}$ represent $\bw\big|_{\Omega_-}$ and $\bw\big|_{\Omega_+}$, respectively. By applying the unique continuation principle and utilizing the fact that $\bu_2$ is real analytic in $B_h$, we derive conditions that lead to a contradiction. Specifically, we obtain
\beq\nonumber
\begin{cases}
\Delta^*_{\lambda_1,\mu_1}\bw^-+\omega^2\bw^-={\bf 0} &\text{in}\quad S_h,\\
\bw=\bu_1\big|_{\Gamma_h^+}-\bu_2\big|_{\Gamma_h^+}=-\bff^1_+ &\text{on}\quad \Gamma^+_h\\
\bw=\bu_1\big|_{\Gamma_h^-}-\bu_2\big|_{\Gamma_h^-}=-\bff^1_-&\text{on}\quad \Gamma^-_h,\\
\mathcal{T}_{\nu}\bw=\mathcal{T}_{\nu}\bu_1\big|_{\Gamma_h^+}-\mathcal{T}_{\nu}\bu_2\big|_{\Gamma_h^+}=-\bg^1_+ &\text{on}\quad \Gamma^+_h\\
\mathcal{T}_{\nu}\bw=\mathcal{T}_{\nu}\bu_1\big|_{\Gamma_h^-}-\mathcal{T}_{\nu}\bu_2\big|_{\Gamma_h^-}=-\bg^1_- &\text{on}\quad \Gamma^-_h,
\end{cases}
\eeq
where $\Delta^*_{\lambda_1,\mu_1}(\cdot)$ denotes the Lam\'e operator in $\Omega^1_1\cap\Omega^2_1$.From Proposition~\ref{prop:trans1}, we conclude that $\bff^1_+({\bf 0})=\bff^1_-({\bf 0})$.These conditions contradict the admissibility condition (iv) in Definition \ref{def:Admis1-3}.

\medskip \noindent In the 3D case, note that ${\bf 0}=({\bf 0}', 0)^t\in \Gamma'^\pm_{h}\times (-M,+M)\Subset\Sigma_1$ is a 3D edge corner point. Additionally, $S'_h\times (-M,+M)\Subset \Omega_-$ and $S'_h\times (-M,+M)\cap\Omega_+=\emptyset$, where $\Gamma'^\pm_{h}$ and $S'_h$ are defined in \eqref{eq:ball1}, and $\Omega_-,\,\Omega_+$ are described as $\rm (i)$ in Definition \ref{def:Admis1-2}. Similar to the 2D case, we obtain
\beq\nonumber
\begin{cases}
\Delta^*_{\lambda_1,\mu_1}\,\bw^-+\omega^2\bw^-=\mathbf{0} &\mbox{in}\quad S'_h\,\,\,\times (-M,M),\\
\bw=\bu_1\big|_{\Gamma_h^+}-\bu_2\big|_{\Gamma_h^+}=-\bff^1_+&\mbox{on}\quad \Gamma'^+_h\times (-M,M),\\
\bw=\bu_1\big|_{\Gamma_h^-}-\bu_2\big|_{\Gamma_h^-}=-\bff^1_-&\mbox{on}\quad \Gamma'^-_h\times (-M,M),\\
\mathcal{T}_{\nu}\bw=\mathcal{T}_{\nu}\bu_1\big|_{\Gamma_h^+}-\mathcal{T}_{\nu}\bu_2\big|_{\Gamma_h^+}=-\bg^1_+ &\mbox{on}\quad \Gamma'^+_h\times (-M,M),\\
\mathcal{T}_{\nu}\bw=\mathcal{T}_{\nu}\bu_1\big|_{\Gamma_h^-}-\mathcal{T}_{\nu}\bu_2\big|_{\Gamma_h^-}=-\bg^1_- &\mbox{on}\quad \Gamma'^-_h\times (-M,M).
\end{cases}
\eeq
These results imply that $\bff^1_+({\bf 0})=\bff^1_-({\bf 0})$
as established by Proposition~\ref{prop:trans2}. This contradicts the admissibility condition (iv) in Definition~\ref{def:Admis1-3}.

\medskip  \noindent {\it Step II:} Establish the local uniqueness of the interfaces (i.e., $\Gamma^1_1 \Delta \Gamma^2_1,\ \Gamma^1_2 \Delta \Gamma^2_2,\ \ldots, \Gamma^1_{N-1} \Delta \Gamma^2_{N-1}$ cannot possess corners or 3D edge corners).

\medskip \noindent Assume that there exists a planar corner or a 3D edge corner $\bx_0$ in $\Gamma^1_1 \Delta \Gamma^2_1$ such that $\bx_0\notin\overline{\Sigma}$. Without loss of generality, we assume that $\bx_0={\bf 0}\in \Gamma^2_1\backslash\Gamma^1_1$ and $\bx_0={\bf 0}\notin  \Gamma^1_1\backslash\Gamma^2_1$. This implies that $\bx_0$ is located in the interior of $\Omega^1_1$.  We then obtain the following PDE system in 2D,
\begin{equation}\nonumber
\begin{cases}
\Delta_{\lambda_1,\,u_1}^*\bu_1+\omega^2\,\bu_1={\bf 0} &\mbox{in}\, \,S_h,\\
\Delta_{\lambda_2,\,u_2}^*\bu_2+\omega^2\,\bu_2={\bf 0} &\mbox{in} \,\,S_h,\\
\bu_1=\bu_2,\,\,\mathcal{T}^1_{\nu}\bu_1=\mathcal{T}^2_{\nu}\bu_2&\mbox{on}\,\Gamma^\pm_h,
\end{cases}
\end{equation}
or the following PDE system in 3D
\begin{equation}\nonumber
\begin{cases}
\Delta_{\lambda_1,\,u_1}^*\bu_1+\omega^2\,\bu_1={\bf 0} &\mbox{in} \,\,S'_h\,\,\,\times (-M,M),\\
\Delta_{\lambda_2,\,u_2}^*\bu_2+\omega^2\,\bu_2={\bf 0} &\mbox{in} \,\,S'_h\,\,\,\times (-M,M),\\
\bu_1=\bu_2,\,\, \mathcal{T}^1_{\nu}\bu_1=\mathcal{T}^2_{\nu}\bu_2&\mbox{on}\,\,\Gamma'^\pm_h\times (-M,M),
\end{cases}
\end{equation}
where $S_h$, $S'_h$, $\Gamma^\pm_h$ and $\Gamma'^\pm_h$ are defined in \eqref{eq:ball1}.

\medskip \noindent In the 2D case, utilizing the CGO solution $\widetilde{\bu}_0$ defined in \eqref{eq:lame3}, which satisfies $\Delta^*_{\lambda_2,\mu_2}\bu=0$ in $S_h$, and conforms to Betti's second formula, we have the following integral identity
\begin{align*}
\int_{\Gamma^\pm_h}\left(\mathcal{T}^2_\nu\bu_1-\mathcal{T}^1_\nu\bu_1\right) \cdot \widetilde{\bu}_0\,\mathrm{d}\sigma=&-
\int_{\Lambda_h}(\mathcal{T}^2_{\nu}\bu_1-{\mathcal{T}^2_{\nu}\bu_2)\cdot \widetilde{\bu}_0-\mathcal{T}^2_{\nu}}\widetilde{\bu}_0\cdot(\bu_1-\bu_2)\,\mathrm{d}\sigma\\
&+\int_{S_h}\Delta^*_{\lambda_2,\mu_2}\bu_1\cdot \widetilde{\bu}_0\,\mathrm{d}\bx+\int_{S_h}\omega^2\bu_2\cdot \widetilde{\bu}_0\,\mathrm{d}\bx
\end{align*}
It is important to note that $\bu_1 = (u_1, u_2)^t$ is analytic in $\overline{S}_h$. This implies that $\Delta^* \bu_1$ is also analytic. Furthermore, we have
$$\mathcal{T}^2_\nu\bu_1-\mathcal{T}^1_\nu\bu_1:=\bt(\bx)=\bt({\bf 0})+\delta \bt(\bx),$$
where $\delta\bt(\bx)$ satisfies the bound $ |\delta\bt|\leq C_0|\bx|^\zeta
$ for some constant $C_0 > 0$ and $\zeta > 0$.
From Lemma \ref{10-lem:23} and Lemma \ref{10-lem:u0 int}, the following estimates hold,
\begin{align*}
&\left|\int_{\Lambda_h}(\mathcal{T}^2_{\nu}\bu_1-{\mathcal{T}^2_{\nu}\bu_2)\cdot \widetilde{\bu}_0-\mathcal{T}^2_{\nu}}\widetilde{\bu}_0\cdot(\bu_1-\bu_2)\,\mathrm{d}\sigma\right|\leq C_1  e^{-\widetilde{D}_1s },\\[2pt]
&\left|\int_{S_h}\Delta^*_{\lambda_2,\mu_2}\bu_1\cdot \widetilde{\bu}_0\,\mathrm{d}\bx\right|
\leq C_2 s^{-4},\,\, \left|\int_{S_h}\omega^2\bu_2\cdot \widetilde{\bu}_0\,\mathrm{d}\bx\right|\leq C_3  e^{-\widetilde{D}_2s },\\
&\left|\int_{\Gamma^\pm_h} \delta\bt({\bx})\cdot \widetilde{\bu}_0\,\mathrm{d}\sigma \right|\leq C_4 s^{-2\zeta-2},\\
& \int_{\Gamma^\pm_h} \bt({\bf 0})\cdot \widetilde{\bu}_0\,\mathrm{d}\sigma=\bt({\bf 0})\cdot \begin{bmatrix}
    1\\ \mathrm{i}
\end{bmatrix}2s^{-2}\Big(\hat{\mu}^{-2}(\theta_{\max})+\hat{\mu}^{-2}(\theta_{\min})\Big)+\Oh(s^{-1}e^{-\widetilde{D}_4 s}),
\end{align*}
where $C_1,C_2, C_3,\widetilde{D}_1$,$\widetilde{D}_2$, $\widetilde{D}_3$ and $\zeta$ are positive and do not depend on the parameter $s$.
The estimates above lead to the following equation when $s$ goes to $\infty$,
$$
\bt({\bf 0})\cdot \begin{bmatrix}
    1\\ \mathrm{i}
\end{bmatrix}\Big(\hat{\mu}^{-2}(\theta_{\max})+\hat{\mu}^{-2}(\theta_{\min})\Big)=0.
$$
Given that $\hat{\mu}^{-2}(\theta_{\max})+\hat{\mu}^{-2}(\theta_{\min})\neq 0$, it follows directly that $\bt({\bf0})={\bf 0}$.
This leads to a contradiction because, for any interior point $\bx$ of $\Omega^1_1$, it follows from \eqref{eq:lamlayer} that
\begin{equation*}
\mathcal{T}^1_\nu\mathbf{u}_1(\bx)=\lambda_1(\nabla\cdot\mathbf{u}_1(\bx))\nu+2\mu_1(\nabla^s\mathbf{u}_1(\bx))\neq\lambda_2(\nabla\cdot\mathbf{u}_1(\bx))\nu+2\mu_2(\nabla^s\mathbf{u}_1(\bx))=\mathcal{T}^2_{\nu}\bu_1(\bx),
\end{equation*}
Here, $\nu$ represents any unit vector, while $\mathcal{T}^1_{\nu}\bu_1$ and $\mathcal{T}^2_{\nu}\bu_1$ denote the traction fields corresponding to the Lam\'e parameters $\lambda_1,\mu_1$ and $\lambda_2,\mu_2$, respectively. Therefore, we conclude that $\Gamma^1_1\Delta\Gamma^2_1$ cannot contain a corner. Using the same method as described above, we can prove that $\Gamma^1_i\Delta\Gamma^2_i$ cannot contain a plane corner for $i=2,3,..., N-1$.

Similarly, the same method can be applied to demonstrate that edge corners cannot exist on $\Gamma^1_1\Delta\Gamma^2_1$, $\Gamma^1_2\Delta\Gamma^2_{2},\ldots,\Gamma^1_{N-1}\Delta\Gamma^2_{N-1}$ in the 3D case.

\medskip \noindent In {\bf Case II},we consider the scenario where the planar corner or 3D edge corner of the fault $\Sigma$  is located in the subdomain $\Omega_k$
  closest to $\partial\Omega_0$, with $k\geq 2$ and $k\in \mathbb{N}$. By applying the same method as in {\bf Case I}, we can establish that corners cannot exist on $\Gamma^1_1\Delta\Gamma^2_1$, $\Gamma^1_2\Delta\Gamma^2_2,\ldots,\Gamma^1_{k-1}\Delta\Gamma^2_{k-1}$. Next, we extend this method to address a planar or 3D edge corner denoted as $\bx_c$ located on $\Sigma_1\Delta\Sigma_2$ within the domain $\Omega_k$. Similarly, we can prove that $\Sigma_1\Delta\Sigma_2$ cannot contain a planar or 3D edge corner. By applying the method from  {\bf Case I} once again, we can establish that edge corners cannot exist on $\Gamma^1_k\Delta\Gamma^2_k$, $\Gamma^1_{k+1}\Delta\Gamma^2_{k+1},...,\Gamma^1_{N-1}\Delta\Gamma^2_{N-1}$.

The proof is complete.
\end{proof}

\begin{proof}[Proof of Theorem~\ref{th:main1}] To prove that $\Sigma_1=\Sigma_2$ by contradiction, we assume that
$\Omega^1_-\neq\Omega^2_-$. Since both $\Omega^1_-$ and $\Omega^2_-$ are convex polygons or polyhedrons, there must exist a corner $\bx_c$ belonging to $\Omega^1_-\Delta\Omega^2_-$ must exist. However, this contradicts Theorem~\ref{th:main_loca}. Therefore, we conclude that $\Omega^1_-=\Omega^2_-$ implies $\Sigma_1=\Sigma_2$. Furthermore, it follows from Theorem~\ref{th:main_loca} and the fact that the interfaces are piecewise polygonal that
 $\Gamma^1_1=\Gamma^2_1$, $\Gamma^1_{2}=\Gamma^2_{2},..,\Gamma^1_{N-1}=\Gamma^2_{N-1}$.

\medskip \noindent In the 2D case, we denote $\hat{\Omega}:=\Omega^1_-=\Omega^2_-$ and $\Sigma:=\Sigma_1=\Sigma_2$. Let $\bw=\bu_1-\bu_2$. Since $\bu_1=\bu_2$ on $\partial\Omega_0\Subset\partial\Omega_2$, we have $\bw=0$ and $\mathcal{T}_{\nu}\bw=0$ on $\Sigma_0$. By applying the unique continuation principle, we obtain
$$
\bw^+\big|_{\Gamma_h^\pm}=\mathcal{T}_{\nu}\bw^+\big|_{\Gamma_h^\pm}={\bf 0},
$$
where $\bw^{+}$ represents $\bw\big|_{\Omega\backslash\overline{\hat{\Omega}}}$. Since $(\Sigma;\bff^1,\bg^1)$ and $(\Sigma;\bff^2,\bg^2)$ are admissible, we have
\beq\label{eq:c1}
\begin{aligned}
\Delta^*\bw^-+\omega^2\,\bw^-={\bf 0}\quad \text{in}\quad S_h,\
\bw^-\big|_{\Gamma_h^j}=\bff^2_j-\bff^1_j,\quad \mathcal{T}_{\nu}\bw^-\big|_{\Gamma_h^j}=\bg^2_j-\bg^1_j\quad \text{for}\quad j=+,-,
\end{aligned}
\eeq
where $\bw^{-}$ denotes $\bw\big|_{\hat{\Omega}}$.
From Proposition \ref{prop:trans2} and \eqref{eq:c1}, we obtain the following local uniqueness
$$
\bff^2_+({\bf 0})-\bff^1_+({\bf 0})=\bff^2_-({\bf 0})-\bff^1_-({\bf 0}).
$$
Since $\bff^i$ and $\bg^i$ ($i=1,2$) are piecewise-constant functions, we conclude that \eqref{eq:unque1'} holds.

The same method can be applied to prove \eqref{eq:unque1} and \eqref{eq:unque1'} in the 3D case. Finally, we conclude the proof.
\end{proof}

\begin{proof}[Proof of Theorem~\ref{th:main2}]
The proof of this theorem follows a similar argument to that in the proof of Theorem~\ref{th:main1}, with necessary modifications. Consider two different linear piecewise curves $\Sigma_1$ and $\Sigma_2$ in $\mathbb R^2$ or polyhedral surfaces in $\mathbb R^3$. From Definition \ref{def:poly},  From Definition \ref{def:poly}, it is clear that $\Sigma_1 \Delta \Sigma _2$ contains a planar or 3D edge corner.

Under the condition \eqref{eq:thm33}, by adopting an argument similar to that used in the proof of Theorem \ref{th:main1}, we can show that $\Sigma_1= \Sigma_2$. Furthermore, we can establish that $\Gamma^1_1=\Gamma^2_1$, $\Gamma^1_{2}=\Gamma^2_{2},\ldots,\Gamma^1_{N-1}=\Gamma^2_{N-1}$.

Let $\bw=\bu_1-\bu_2$. We have
$\bw=\mathcal{T}_{\nu}\bw=\bf 0$ on $\partial\Omega_0$ because
 $\bu_1=\bu_2$ on $\partial\Omega_0\Subset\partial\Omega_2$. By applying the unique continuation property, we conclude that $\bw=\bf 0$ in $\Omega\backslash \overline{\Sigma}$. Hence, it follows directly that
$$
\bf 0=[\bw]_{\Sigma_1}=[\bw]_{\Sigma_2}\,\Rightarrow\, [\bu_1]_{\Sigma_1}=[\bu_2]_{\Sigma_2}\,\,\mbox{and}\,\,[\mathcal{T}_{\nu}\bu_1]_{\Sigma_1}=[\mathcal{T}_{\nu}\bu_2]_{\Sigma_2}\,\Rightarrow\, \bff^1=\bff^2 \,\mbox{and}\,\bg^1=\bg^2.
$$

The proof is complete.
\end{proof}

\par

\noindent\textbf{Acknowledgment.}
The work of H. Diao is supported by the National Natural Science Foundation of China  (No. 12371422) and the Fundamental Research Funds for the Central Universities, JLU (No. 93Z172023Z01).  The work of H. Liu is supported by the Hong Kong RGC General Research Funds (projects 11311122, 11300821, and 12301420), NSF/RGC Joint Research Fund (project N\_CityU101/21), and the ANR/RGC Joint Research Fund (project A\_CityU203/19). The work of Q. Meng is fully supported by a fellowship award from the Research Grants Council of the Hong Kong Special Administrative Region, China (Project No. CityU PDFS2324-1S09).

\end{document}